%% file: main.tex
\documentclass[pdflatex,sn-mathphys-num]{sn-jnl}

\usepackage{svg}
\usepackage{graphicx}
\usepackage{natbib}
\usepackage{amsmath,mathtools,amssymb,breqn,amsthm}
\usepackage{xcolor}
\usepackage{ifthen}
\usepackage{subcaption}
\usepackage{caption}
\usepackage{siunitx}
\usepackage{algorithm, algpseudocode}
\usepackage{doi}
\usepackage[]{hyperref}

\usepackage[normalem]{ulem}
\newboolean{showcalculations}
\setboolean{showcalculations}{false}
\newboolean{showchanges}
\setboolean{showchanges}{true}
\usepackage{pgfplots}
\usepgfplotslibrary{groupplots,dateplot}
\usetikzlibrary{patterns,shapes.arrows}
\pgfplotsset{compat=newest}
\usepgfplotslibrary{external} 
\usepackage{autonum}
\input{macros_text.tex}
\input{macros_math.tex}

\input{macros_todo.tex}

\input{macros_thm.tex}

\begin{document}
	\title[Neural network-enhanced integrators for simulating ordinary differential equations]{Neural network-enhanced integrators for simulating ordinary differential equations}
	
	
	\author*[1]{\fnm{Amine} \sur{Othmane}}\email{amine.othmane@uni-saarland.de}
	
	\author[1]{\fnm{Kathrin} \sur{Flaßkamp}}\email{kathrin.flasskamp@uni-saarland.de}
	
	\affil[1]{Systems Modeling and Simulation, Saarland University, Germany}
	
	
	\abstract{Numerous applications necessitate the computation of numerical solutions to differential equations across a wide range of initial conditions and system parameters, which feeds the demand for efficient yet accurate numerical integration methods.
		This study proposes a neural network (NN) enhancement of classical numerical integrators. 
		NNs are trained to learn integration errors, which are then used as additive correction terms in numerical schemes.
		The performance of these enhanced integrators is compared with well-established methods through numerical studies, with a particular emphasis on computational efficiency. Analytical properties are examined in terms of local errors and backward error analysis. Embedded Runge-Kutta schemes are then employed to develop enhanced integrators that mitigate generalization risk, ensuring that the neural network's evaluation in previously unseen regions of the state space does not destabilize the integrator.
		It is guaranteed that the enhanced integrators perform at least as well as the desired classical Runge-Kutta schemes. The effectiveness of the proposed approaches is demonstrated through extensive numerical studies using a realistic model of a wind turbine, with parameters derived from the established simulation framework OpenFast.
	}

	
	
	
	\maketitle

\section{Introduction}
The need for accurate yet efficient numerical solutions to differential equations remains a cornerstone of modern computational mathematics and engineering. For example, the analysis of fatigue effects and lifetime prediction of complex technological systems, e.g.\ wind turbines, requires repeatedly solving differential equations under a variety of varying initial conditions and system parameters.
The inherently stochastic characteristics of wind conditions encompassing turbulence, shear, and wake interactions necessitate extensive numerical simulations over a wide range of environmental and operational conditions to enable reliable fatigue assessment.
This example demonstrates the challenges faced in numerical integration: the need for high accuracy in approximating solutions while being constrained by computational resources. The traditional approach, which relies heavily on classical numerical methods such as Runge–Kutta methods (e.g., \cite{hairer2008,butcher2016numerical}), often encounters limitations when scaling to high-dimensional systems or when long-term integration is required.

\subsection{State of the art}
In order to overcome these limitations, extrapolation techniques have been developed to enhance both the accuracy and efficiency of numerical solvers. One of the earliest such techniques is Richardson extrapolation \cite{richardson1910,richardson1927,Atkinson1989}, which systematically eliminates the leading error term by combining solutions computed with different step sizes. More generally, extrapolation methods form the basis of adaptive algorithms that refine the solution iteratively using estimates of the local error. In particular, the Bulirsch–Stoer algorithm \cite{StoerBulirsch2002} employs Richardson extrapolation applied to the modified midpoint rule, efficiently estimating the local error and extrapolating to the zero-step-size limit. These methods have proven particularly effective in handling stiff or high-dimensional systems, as they leverage accurate local error estimates to reduce the overall computational effort without sacrificing precision.
Adaptive step-size control further enhances solver performance. Embedded Runge–Kutta pairs \cite{DormandPrince1980} dynamically adjust the time step using error estimates computed during the integration process, ensuring that the local error remains within a specified tolerance while optimizing computational effort.

Neural networks possess remarkable approximation properties, as demonstrated by the universal approximation theorem, which asserts that a neural network with a sufficient number of neurons and appropriate activation functions, such as the rectified linear unit, can under mild assumptions approximate any continuous function to any desired degree of accuracy \cite{Cybenko1989,Hornik1989,Hornik1991,Hanin2019,Yarotsky2017}. These properties have been used  in recent years to develop neural network‐based methods for approximating the solutions of differential equations. In particular, mesh‐free approaches for solving partial differential equations (PDEs) have been developed and shown to outperform traditional solvers in terms of computational speed and flexibility \cite{li2020a,li2020b}. Moreover, hybrid strategies that combine neural networks with classical time‐stepping  low-order methods have been investigated as a means to enhance accuracy and stability in the numerical solution of ordinary differential equations (ODEs) \cite{shen2020,huang2022,huang2023}. The correction of the forward Euler method has been considered and analyzed in \cite{shen2020,huang2022}. Simulations using an enhanced Heun method have been considered in \cite{shen2020}.

\subsection{Contribution}
Taking advantage of the ability of NNs to approximate nonlinear functions, we propose a general methodology in which NNs are trained to learn the integration error of arbitrary order explicit one-step Runge-Kutta methods.
This learned error is then used as an additive correction term for the classical discrete flow to  achieve a better compromise between computational load and accuracy.
Such a hybrid approach aims to combine the physics-based classical numerical techniques with the adaptive learning capabilities of neural networks. 
Our investigation includes an analysis of the theoretical properties of the solvers such as global error, modified differential equation and accuracy-efficiency trade-offs for arbitrary Runge-Kutta methods. 
Moreover, by using embedded Runge-Kutta schemes in an adaptive algorithm, we ensure that outside the validity domain of the network, the proposed solver is not destabilized.
It is guaranteed that the enhanced integrators perform at least as well a user defined classical Runge-Kutta schemes.
The numerical experiments for a non-trivial academic example validate the theoretical findings and confirm that the enhanced integrator can outperform classical methods in the accuracy-efficiency trade-off.

\subsection{Application case study}
The effectiveness of the proposed methods are shows using a wind turbine model developed within the CADynTurb framework \cite{Geisler2021,Geisler2021a,Geisler2024} with 16 state components. Designed to capture the nonlinear, multi-degree-of-freedom dynamics of wind turbines, the model incorporates eight generalized coordinates that represent key structural and electromechanical components—including tower displacements, blade flapwise and edgewise bending modes, and drivetrain as well as generator dynamics. Derived from the Euler–Lagrange equations and parameterized using publicly available data such as the NREL 5-MW reference turbine \cite{Jonkman2009}, this model is capable of simulating the complex cyclic loading conditions essential for fatigue analysis \cite{Shaler2023,Ding2024,Liu2023,Schafhirt2022}. Its balance between physical fidelity and computational efficiency makes it ideally suited for extensive simulations required in design optimization and real-time estimation tasks.

\subsection{Outline} The remainder of this paper is structured as follows. Section~\ref{sec:Problem_statement} defines the problem statement and recalls necessary background material. Enhanced integrators are proposed in Section~\ref{sec:enhanced_integrators}. First, a simple additive correction is proposed and analyzed with respect to local error approximation and backward error analysis. Then, embedded Runge-Kutta schemes are used to derive an approach that mitigates generalization risks.
In Section~\ref{sec:numerical_results}, the effectiveness of the approach is analyzed in extensive numerical case studies.
Concluding remarks are given in Section~\ref{sec:Conclusion}.

\section{Problem statement and background}
\label{sec:Problem_statement}
Let $t_{\mathrm{end}}$ be some scalars such that $t_{\mathrm{end}}>0$ and denote by $\setR_{\geq 0}$ the set of all real numbers greater than or equal to $0$. Consider the set $\setStateTime \subset\setR^{\dimensionState}$, $\dimensionState\in\setN$ and  let $ \SystemState_0$ be an element of ${\setStateTime}$.
 Let $\setParam$ be a subset of $\setR^{\dimensionParam}$, $\dimensionParam\in\setN$, with $\parameter_0 \in \setParam$. In the following, the initial value problem 
\begin{equation}
	\dot{\SystemState}(t) = f( \SystemState(t), \parameter_0), \quad \SystemState(0) = \SystemState_0\in\setR^\dimensionState,
	\label{eq:ODE}
\end{equation}
is considered for ${t\in[0,t_{\mathrm{end}}]}$, where the vector field ${f: \setStateTime \times \setParam \rightarrow \setR^{\dimensionState}}$ is assumed to be   Lipschitz continuous with respect to its first argument.
We define the \textit{flow} ${\flowCont:\setR^{\dimensionState}\times\setR^{\dimensionParam}\rightarrow\setR^{\dimensionState}}$ to be the function that associates, for a given time $t$, to the initial value  $\SystemState_0$ at time zero and parameter $\parameter_0$ the corresponding solution of \eqref{eq:ODE} at time $t$. 
Thus, this map is defined by
\begin{equation}
	\flowCont(\SystemState_0,\parameter_0)=\SystemState(t) \quad\text{with}\quad \SystemState(0)=\SystemState_0,
\end{equation}
with $\SystemState$ solving \eqref{eq:ODE}.

For most initial value problems, analytical solutions are unattainable.
Then, the flow of \eqref{eq:ODE} can be approximated by numerical schemes. Explicit one-step schemes, typically characterized in Butcher tableaux, are a common choice (see, e.g. \cite{butcher2016numerical,hairer2008}).
For a given step size $\stepsize$, it follows that $\SystemState(\iterator\stepsize)\approx\approximation{\SystemState}_{\iterator}$ where
the latter is the result of the recursion 
\begin{equation}
	\approximation{\SystemState}_{\iterator+1}=\flowDiscrete(\approximation{\SystemState}_{\iterator}),\quad \forall \iterator\geq 0,
	\label{eq:explicit_num}
\end{equation}
such that $\approximation{\SystemState}_{0}=\SystemState_0$, where the map $\flowDiscrete:\approximation{\SystemState}_{\iterator}\mapsto \approximation{\SystemState}_{\iterator+1}$ is called the \textit{discrete flow} (see for instance \cite{hairer2008,butcher2016numerical}).

\begin{exmp}
	\label{exmp:Euler}
	The simplest of all numerical schemes is the forward Euler method. 
	The discrete flow is then given as 
	\begin{equation}
		\flowDiscrete(\SystemState) = \SystemState + \stepsize f(\SystemState, \parameter_0).
	\end{equation}	
\end{exmp}
\begin{exmp}
    \label{exmp:Heun}
	A second-order improvement is provided by the explicit trapezoidal rule, also known as Heun's method. It evaluates the vector field at both the initial and predicted states such that $\flowDiscrete(\SystemState) = \SystemState + \frac{\stepsize}{2}(k_1(\SystemState) + k_2(\SystemState))$ with
	\begin{align}
		k_1(\SystemState) &= f(\SystemState, \parameter_0), \\
		k_2(\SystemState) &= f(\SystemState + \stepsize k_1, \parameter_0).
	\end{align}
\end{exmp}

In the sequel, the error $\flowDiscrete(\SystemState)-\flowCont[(\iterator+1)\stepsize](\SystemState,\parameter_0)$, for arbitrary $\SystemState\in\setR^\dimensionState$, incurred when integrating the differential equation over a single step of size $\stepsize$, starting from the initial condition $\SystemState((\iterator-1) h)$, is called \emph{local error} (see, e.g., \cite{hairer2008}).
A numerical scheme is said to be of order $\integratorOrder\in\setN$, if the local error is proportional to $\stepsize^{\integratorOrder+1}$ (see, e.g., \cite{hairer2008}).
It is known that reducing the step size or using a high-order integration scheme can reduce the error. However, this yields a critical trade-off between computation speed and accuracy. Different approaches dealing with this trade-off such as extrapolation techniques or variable step sizes have also been developed in the literature \cite{richardson1910,richardson1927,Atkinson1989,StoerBulirsch2002,butcher2016numerical,hairer2008}.

\paragraph*{Problem statement:}
This contribution analyzes the capabilities of artificial neural networks  to enhance classical numerical integrators with the goal of achieving a better accuracy versus efficiency compromise for solving initial value problems of  type \eqref{eq:ODE} across a wide range of initial conditions $\SystemState_0$ and system parameters $\parameter_0$ in the interval $[0,t_{\text{end}}]$.

\section{Neural network enhanced integrators}
\label{sec:enhanced_integrators}

In this section, we outline a method to correct numerical integration errors by training artificial neural networks to learn these discrepancies.
 The error approximation is then used as an additive correction term to the discrete flow.
  Then, enhanced integrators with mitigated generalization risk are proposed.
   The theoretical properties of the approaches are analyzed.

\subsection{Definition of neural networks}
Consider a Lipschitz continuous activation function
$\NNactivation:\setR\rightarrow\setR$ and define for ${\SystemState\in\setR^{n}}$ the function ${\NNactivationStuck:\SystemState\mapsto \myVector{\NNactivation(x_1),\hdots,\NNactivation(x_{n})}^T}$, where $x_j$, $j\in\{1,\hdots,n\}$, is the $j$-th component of $\SystemState$. An artificial neural network of depth $\NNdepth\in\setN$ is a function $\NN:\setR^{\NNdimIn}\rightarrow\setR^{\NNdimOut}$, $\NNdimIn,\NNdimOut\in\setN$, that maps $\SystemState\in\setR^{\NNdimIn}$ to
\begin{equation}
\NN(\SystemState)=T_\NNdepth\NNactivationStuck\left(T_{\NNdepth-1}\NNactivationStuck\left(\hdots \NNactivationStuck(T_1(\SystemState))\right)\right),
\end{equation}
where $T_l$, $l\in\{1,\hdots,\NNdepth\}$, is an affine-linear transformation defined as
\begin{equation}
	T_l:\setR^{n_{l-1}}\rightarrow\setR^{n_l},\; T_lx=W^{(l)}x+b^{(l)}, \; n_{0}=\NNdimIn,\; n_{\NNdepth}=\NNdimOut,
\end{equation}
with ${W^{(l)}\in\setR^{n_l\times n_{l-1}}}$ the weight matrices and ${b^{(l)}\in\setR^{n_l}}$ the bias vector of the $l$-th layer all of which are summarizes in $\theta$ such that ${\NNparam=\left\{(W^{(l)},b^{(l)})\right\}_{l=1}^L}$. Denote in the following by $\NNsetParam$ the set of all possible ${(W^{(l)},b^{(l)})}$, \ie, 
$$\NNsetParam=\left\{(W^{(l)},b^{(l)})\in\setR^{n_l\times n_{l-1}}\times \setR^{n_l}|l\in\{1,\hdots,\NNdepth\}\right\}.$$
\subsection{Learning problem for local error correction}\label{sec:NN_correction}
Let  ${\mathcal{K}}$ be a finite subset of $\setN$ and  consider the set ${\NNtrainingDataIn = \{(\iterator_i,\SystemState_{0i},\parameter_i)\in\mathcal{K}\times\setStateTime\times\setParam\,|\,i\in\{1,\hdots,{\NNcarTrainingData} 
\}\}}$ of  finite cardinality $\NNcarTrainingData$ with time steps, initial conditions, and parameters. For each element in $\NNtrainingDataIn$ define for a numerical scheme with discrete flow $\flowDiscrete$ the local error $r:\NNtrainingDataIn\rightarrow\NNtrainingDataOut\subset\setR^\dimensionState$ that maps each element $(\iterator,\SystemState_0,\parameter_0)$ to
\begin{equation}
	r(\iterator,\SystemState_0,\parameter_0)={\flowCont[(\iterator+1)\stepsize](\SystemState_0,\parameter_0)-\flowDiscrete\!\left(\!\flowCont[\iterator\stepsize](\SystemState_0,\parameter_0)\!\right)}.
	\label{eq:local_error}
\end{equation}
We now employ a neural network, denoted by $\NN$, to construct an approximation of the local error \eqref{eq:local_error}, with the aim of its subsequent incorporation into an adapted numerical scheme for error compensation. 
The parameter set $\NNparam$ of $\NN$ is computed using the flow $\flowCont$ or a high order approximation of it evaluated at different times and for various initial states and parameter configurations. The generation of the data can be done using exact analytical solutions, high-fidelity numerical solvers, or empirical data. Subsequent validation against simulated trajectories reveals that a small number of trajectories suffice to yield satisfactory approximations. This is discussed in detail in Section \ref{sec:numerical_results}.

The parameters $\NNparam$ of $\NN$ are then computed as the solution of an optimization problem
\begin{subequations}\label{eq:NN_Training}
	\begin{equation}
		\NNparam=\argmin_{\tilde{\NNparam}\in\NNsetParam}\frac{1}{\NNcarTrainingData}\sum_{(\iterator,\SystemState_0,\parameter_0)\in\NNtrainingDataIn}\NNloss_{\tilde{\NNparam}}(\iterator,\SystemState_0,\parameter_0)
		\label{eq:NN_OP}
	\end{equation}
	with the function  $\NNloss_{\tilde{\NNparam}}$ given by
	\begin{equation}
		\begin{aligned}
			\NNloss_{\tilde{\NNparam}}(\iterator,\SystemState_0,\parameter_0)=&\Big\lVert\NN\bigl(\flowCont[\iterator\stepsize](\SystemState_0,\parameter_0),\parameter_0\bigr)
			-\tfrac{1}{\stepsize^{\integratorOrder+1}}r\!\bigl(\iterator,\SystemState_0,\parameter_0\bigr)\Big\rVert^2,
		\end{aligned}
		\label{eq:NN_loss}
	\end{equation}
\end{subequations}
to approximate the local error, where $\norm{\cdot}$ is some norm. 

Note that the loss function \eqref{eq:NN_loss} depends on the step size $\stepsize$ and thus classical approximation arguments for numerical integrators can be applied for the subsequent analysis.
For a consistent integrator, there exists a function $\delta$ such that the local truncation error is equal to $\stepsize^{p+1}\delta(\SystemState_0)+\mathcal{O}\left(\stepsize^{p+2}\right)$ \cite{hairer2002b}. 
Because the optimization problem \eqref{eq:NN_OP} uses samples 
\(\SystemState_0\) in the entire region spanned by the training data, 
the network thus approximates \(\delta(\SystemState_0)\) across that 
entire state domain. Thus, when the loss function \eqref{eq:NN_loss} of the optimal solution is of order $\mathcal{O}\left(h^{p+2}\right)$ the approach 
provides first-order approximation in the step size \(h\) to the 
local error within that region.

\subsection{Enhanced solvers}
Once the training of the neural network successfully converged to a good  minimizer of \eqref{eq:NN_OP},
the solution of \eqref{eq:ODE} can then be approximated using an enhanced numerical scheme defined by
\begin{equation}
	\approximation{\SystemState}_{\iterator+1}=\flowDiscrete(\approximation{\SystemState}_{\iterator})+\stepsize^{\integratorOrder+1}\NN(\approximation{\SystemState}_{\iterator},\parameter_0),
	\label{eq:enhanced_explicit_num} 
\end{equation}
for all $ \iterator\geq 0$ and $\approximation{\SystemState}_{0}=\SystemState_0$.

Using an approximation of the local error to enhance numerical schemes is a well known approach in the classical numerical integration literature. A well known method is the Richardson extrapolation developed in \cite{richardson1910,richardson1927} (see also \cite[Ch.~II.4]{hairer2008}).
Consider a numerical approximation $\approximation{\SystemState}_{\iterator}$ of the solution $\flowCont(\SystemState_0, \parameter_0)$ of the initial value problem \eqref{eq:ODE} obtained using a numerical scheme with step size $\stepsize$. The Richardson extrapolation technique aims to improve the accuracy of this approximation by combining approximations with different step sizes. Let $\stepsize_1$ and $\stepsize_2$ be two different step sizes such that ${\stepsize_2 = \frac{\stepsize_1}{2}}$. Denote by $\approximation{\SystemState}_{\iterator}(\stepsize_1)$ and $\approximation{\SystemState}_{\iterator}(\stepsize_2)$ the approximations obtained with step sizes $\stepsize_1$ and $\stepsize_2$, respectively. The Richardson extrapolated solution $\approximation{\SystemState}_{\iterator}^{\mathrm{rich}}$ is given by 
\begin{equation}
	\approximation{\SystemState}_{\iterator}^{\mathrm{rich}} = \frac{2^p \approximation{\SystemState}_{\iterator}(\stepsize_2) - \approximation{\SystemState}_{\iterator}(\stepsize_1)}{2^p - 1}, \label{eq:richardson} 
\end{equation} 
where $p$ is the order of the numerical scheme used. This extrapolation effectively cancels out the leading error term, resulting in a more accurate approximation of the true solution.
	
The enhanced numerical scheme \eqref{eq:enhanced_explicit_num} resembles the idea of extrapolations methods. However, approximating the local truncation error using neural networks instead of the Richardson extrapolation method, for example, yields more design freedom. This extra flexibility given by possible variations of the cost function in \eqref{eq:NN_Training} permits the incorporation of tailored regularization terms, domain-specific weighting, or alternative error norms into the training process, allowing the network to more accurately capture and correct the local error structure. Consequently, such an approach can lead to enhanced robustness and improved generalization of the error correction compared to fixed, classical extrapolation techniques.

\subsection{Analysis of enhanced integrators}
\begin{prop}
	Consider the initial value problem \eqref{eq:ODE} and an artificial neural network
	$
	\NN : \setR^{\dimensionState} \times \setR^\dimensionParam \;\to\;\setR^{\dimensionState}.
	$
	Let $\flowDiscrete$ be a discrete flow of a numerical scheme of order $\integratorOrder$ with step size $\stepsize$, 
	and denote by $L_{\flowDiscrete}$ a Lipschitz constant of it.
	Let $L_\NN$ be a Lipschitz constant of $\NN$ 
	and $\NNerror>0$ a scalar such that
	\begin{equation}
		\bigl\lVert r(\iterator,\SystemState_0,\parameter_0)
		-\stepsize^{\integratorOrder+1}\,\NN(\flowCont[\iterator\stepsize](\SystemState_0,\parameter_0),\parameter_0)\bigr\rVert
		\;<\;\NNerror\,\stepsize^{\integratorOrder+1},\quad k\in\setN,\;\SystemState_{0}\in\setStateTime,\;\parameter_0\in\setParam,
	\end{equation}
	with the local error $r$ defined in \eqref{eq:local_error}.
	
	Consider the error
	$$
	\integratorError_{\iterator+1} 
	\;=\;
	\flowDiscrete\bigl(\approximation{\SystemState}_{\iterator}\bigr)
	\;+\;\stepsize^{\integratorOrder+1}\,\NN\bigl(\approximation{\SystemState}_{\iterator},\parameter_0\bigr)
	\;-\;\flowCont[(\iterator+1)\stepsize](\SystemState_0,\parameter_0),
	$$
	with $\approximation{\SystemState}_{\iterator}$ defined by \eqref{eq:enhanced_explicit_num}.
	Then, the bound
	\begin{equation}
		\label{eq:error_bound}
		\norm{\integratorError_{\iterator+1}}
		\;\leq\;
		\begin{cases}
			\displaystyle
			\frac{\exp\!\Bigl((\iterator+1)\,\bigl(\alpha - 1\bigr)\Bigr) \;-\;1}
			{\alpha - 1}\,\NNerror \,\stepsize^{\integratorOrder+1},
			&\text{if } \alpha > 1,\\
			(\iterator+1)\,\NNerror \,\stepsize^{\integratorOrder+1},
			&\text{if } \alpha = 1,\\
			\dfrac{\NNerror \,\stepsize^{\integratorOrder+1}}{\,1-\alpha\,},
			&\text{if } 0 < \alpha < 1,
		\end{cases}
	\end{equation}
with $\alpha = L_{\flowDiscrete} \;+\;\stepsize^{\integratorOrder+1}L_\NN$, is satisfied for all $k\in\setN$, $\SystemState_{0}\in\setStateTime$, and $\parameter_0\in\setParam$. 
\end{prop}

\begin{proof}
	Let for notational simplicity $\SystemState_{k}\coloneqq\flowCont[\iterator+1\stepsize](\SystemState_0,\parameter_0)$. 
	Then from the hypothesis on $\NN$, for $k=0$ we have
	\begin{equation}
		\norm{\integratorError_1}
		\;=\;
		\norm{\flowDiscrete( \SystemState_0) 
			\;+\;\stepsize^{\integratorOrder+1}\,\NN(\SystemState_0,\parameter_0) 
			\;-\;\SystemState_1}
		\;<\;
		\NNerror\,\stepsize^{\integratorOrder+1}
		\;\eqqcolon\;\beta.
	\end{equation}
	Using the Lipschitz continuity of both $\NN$ and the discrete flow $\flowDiscrete$, one obtains 
	for any $k \in \mathbb{N}$,
	\begin{equation}
		\norm{\integratorError_{k+1}}
		\;\le\;
		\NNerror\,\stepsize^{\integratorOrder+1}
		\;+\;\Bigl(L_{\flowDiscrete} \;+\;\stepsize^{\integratorOrder+1}L_\NN\Bigr)
		\,\norm{\integratorError_k}
		\;=\;
		\beta \;+\;\alpha\,\norm{\integratorError_k}.
	\end{equation}
	This implies the linear recurrence
	\begin{equation}
		\norm{\integratorError_{k+1}}
		\;\le\;
		\alpha\,\norm{\integratorError_k}
		\;+\;\beta.
	\end{equation}
	By standard induction (unrolling this recurrence), we get
	\begin{equation}
		\norm{\integratorError_{k+1}}
		\;\le\;
		\beta \,\sum_{j=0}^k \alpha^j.
	\end{equation}
	We now distinguish three cases:
	\begin{enumerate}
		\item[(i)] 
		$\alpha > 1$:
		Recalling that	$\alpha^{k+1} \,\le\, e^{(k+1)(\alpha-1)}$  yields
		\begin{equation}
			\sum_{j=0}^k \alpha^j
			\;=\;
			\frac{\alpha^{k+1}-1}{\alpha-1}
			\;\le\;
			\frac{\exp\!\bigl((k+1)(\alpha-1)\bigr)\;-\;1}{\alpha-1}.
		\end{equation}
		Thus,
		\begin{equation}
			\norm{\integratorError_{k+1}}
			\;\le\;
			\beta \;\frac{\exp\!\bigl((k+1)(\alpha-1)\bigr) -1}{\alpha-1},
		\end{equation}
		as in the first line of \eqref{eq:error_bound}.
		
		\item[(ii)]
		$\alpha = 1$: 
		The geometric sum simply becomes $\sum_{j=0}^k 1 = k+1$.  
		Hence
		\begin{equation}
			\norm{\integratorError_{k+1}}
			\;\le\;
			\beta\,(\,k+1\,),
		\end{equation}
		which is the second line of the piecewise bound.
		
		\item[(iii)]
		$0 < \alpha < 1$:
		Now $\alpha-1 < 0$.  The closed‐form sum is
		\begin{equation}
			\sum_{j=0}^k \alpha^j
			\;=\;
			\frac{1 - \alpha^{k+1}}{\,1-\alpha\,}
			\;\le\;
			\frac{1}{\,1-\alpha\,},
		\end{equation}
		because $0 < \alpha^{k+1}<1$.  Thus
		\begin{equation}
			\norm{\integratorError_{k+1}}
			\;\le\;
			\beta \,\sum_{j=0}^k \alpha^j
			\;\le\;
			\frac{\beta}{\,1-\alpha\,}.
		\end{equation}
		This is exactly the third line of \eqref{eq:error_bound}, 
		and gives a uniform (non‐growing) error bound in this stable regime.
	\end{enumerate}
\end{proof}

Building on the discussions in \cite[Lemma~3.5]{hairer2008}, the adoption of Lipschitz continuity for explicit Runge-Kutta (RK) schemes is justified. A comparison of the error bound in \eqref{eq:error_bound} with those for RK schemes' global error (see \cite[Theorem~3.6]{hairer2008}) reveals that 
thriving for a minimal $\NNerror$ while keeping step size $\stepsize$ fixed, akin to conventional solvers, can significantly minimize the error. Alternatively, reducing $\NNerror$ enables a larger $\stepsize$ without deteriorating accuracy.

For applications, large step sizes decrease computational demands.
Thus, the NN-enhanced approach pays off when the evaluation of the vector field is more time-intensive than vector-matrix multiplications and additions in $\NN$. In Section \ref{sec:numerical_results}, it is shown that enhanced solvers can significantly reduce the computational burden while keeping the numerical accuracy comparable to that of traditional approaches for the considered dynamical system. This can be seen in Figure \ref{fig:comparison_results_median_Error_vs_time} by comparing the results for a given error, i.e., 
comparing the different solutions on horizontal lines of constant accuracy.

The following proposition demonstrates that for a time-invariant differential equation, the enhanced solvers satisfy a modified continuous-time differential equation. This corresponds to backward error analysis an idea that dates back to the work \cite{wilkinson1960error} and is important when the focus is on the qualitative behavior of numerical methods and on making statements valid over very long time intervals.
For simplicity, the dependence on parameters $\parameter$ is omitted.

\begin{prop}
	\label{prop:BEA_Enhanced}
	Consider the sequence ${\approximation{\SystemState}_{k+1}=\flowDiscrete(\approximation{\SystemState}_{k})}$ for $k \geq 0$, where $\flowDiscrete$ is the discrete flow of a consistant integrator of order $p$, i.e., satisfying
	\begin{equation}
		\flowDiscrete(\approximation{\SystemState}_{k})=\flowCont(\approximation{\SystemState}_{k})+\stepsize^{p+1}\delta(\approximation{\SystemState}_{k})+\mathcal{O}\left(\stepsize^{p+2}\right),\label{eq:leading_term_local_error}
	\end{equation}
	where $\flowCont$ denotes the exact flow of ${\dot{\SystemState}=f(\SystemState)}$ and $\stepsize^{p+1}\delta(\approximation{\SystemState}_{k})$ is the leading term of the local truncation error. Consider a neural network $\NN: \SystemState \mapsto \NN(\SystemState)$. An enhanced integrator with discrete flow $\SystemState \mapsto \flowDiscrete(\SystemState) + \stepsize^{p+1}\NN(\SystemState)$ satisfies the modified differential equation
	\begin{equation}
		\dot{\tilde{\SystemState}}=f(\tilde{\SystemState})+\stepsize^{p+1}\left(\delta(\tilde{\SystemState})+\NN(\tilde{\SystemState})\right)+\stepsize^{p+2}f_{p+2}(\SystemState)+\hdots,\label{eq:modified_equation}
	\end{equation}
	where ${\tilde{\SystemState}(0) = \SystemState(0)}$ and the function $f_{p+2}$ depends on $f$ and its derivatives.
\end{prop}

\begin{proof}
	The proof follows immediately from the results in \cite[Ch.~IX.1]{hairer2002b}.
\end{proof}

For a neural network $\NN$ for which the loss function \eqref{eq:NN_loss} is of order $\mathcal{O}\left(h^{p+2}\right)$, the perturbation term of order $h^{p+1}$ in the modified equation \eqref{eq:modified_equation} can be neglected. Furthermore, even if the neural network does not exactly cancel the leading error term but approximates it with an accuracy of $\mathcal{O}\left(h^{p+2}\right)$, the overall error constant is reduced compared to the base integrator with discrete flow $\flowDiscrete$. This systematic improvement demonstrates that the quality of the neural network correction directly governs the effective order of the numerical method. In other words, the enhanced integrator can be interpreted as the exact flow of a perturbed system whose deviation from the original dynamics is of order $\mathcal{O}(h^{p+2})$. Consequently, any long-term error bounds or qualitative properties (such as stability and the preservation of invariants) that hold for the continuous system are inherited by the enhanced integrator, up to errors of order $\mathcal{O}(h^{p+2})$. In particular, if the continuous system preserves a given invariant or exhibits a specific stability behavior, the enhanced integrator will approximately preserve these properties with deviations that remain controlled over long integration times.

\subsection{Enhanced integrators with mitigated generalization risk}
\label{sec:safety_net}
Numerous numerical integration algorithms that adapt the step size to achieve a prescribed tolerance for the local error have been developed (see, e.g., \cite[Ch.~II.4]{hairer2008}) and are part of most modern numerical computing toolboxes. The core idea of these approaches is the use of two Runge--Kutta schemes of order $p$ and $p+1$, respectively, to compute (at least asymptotically) an estimate for the local error. If the error is not within prescribed bounds, the step size is adjusted to ensure both computational efficiency and desired tolerances. This idea is used in the following to mitigate the risk of large NN prediction errors and to only trust the NN if its prediction roughly agrees with prescribed tolerances.

Consider in the following two numerical flows ${}_{p+1}\flowDiscrete$ and ${}_p\flowDiscrete$ for two integrators of orders ${p+1}$ and $p$, respectively.

\subsubsection{Local error approximation:}
An asymptotic estimate for the local error of the less precise method is
\begin{equation}
	{}_{p}\boldsymbol{\epsilon}(\SystemState_k) = {}_{p+1}\flowDiscrete(\SystemState_k) - {}_{p}\flowDiscrete(\SystemState_k).
	\label{eq:error_p_p+1}
\end{equation}
Denote by \( (\SystemState_{k})_i \) the value of the \( i \)-th component of $\SystemState_{k}$. As in automatic step size control algorithms, define a scaling factor
\begin{equation}
	\text{sc}_i
	\;=\;
	\mathrm{Atol}_i
	\;+\;
	\max\bigl(|(\SystemState_{k})_i|,\;|({}_{p+1}\flowDiscrete(\SystemState_{k}))_i|\bigr)\,\mathrm{Rtol}_i,
	\label{eq:scaling}
\end{equation}
which accounts for absolute and relative tolerances prescribed by $\mathrm{Atol}_i$ and $\mathrm{Rtol}_i$, respectively. Relative errors are considered for $\mathrm{Atol}_i=0$ and absolute errors for $\mathrm{Rtol}_i=0$.

\subsubsection{Mitigating the generalization risk of neural networks:}
Let $\NN$ be a neural network with parameters solving \eqref{eq:NN_Training}, trained to enhance the solver corresponding to ${}_p\flowDiscrete$. The neural network aims to approximate the local truncation error of the classical integrator of order $p$.

The discrepancy between the true local error and the neural network prediction can be assessed by considering
\begin{equation}
	\boldsymbol{\delta}(\SystemState_{k},\parameter_0) = {}_{p}\boldsymbol{\epsilon}(\SystemState_k) - \stepsize^{\integratorOrder+1}\NN(\SystemState_{k},\parameter_0).
\end{equation}
To ensure the reliability of the network prediction during generalization, we introduce the normalized error
\begin{equation}
	\tilde{\boldsymbol{\delta}}(\SystemState_{k},\parameter_0) = \left( \tfrac{( \boldsymbol{\delta}(\SystemState_{k},\parameter_0) )_i}{\text{sc}_i} \right)_{i=1}^n,
\end{equation}
where $\text{sc}_i$ denotes a component-wise scaling factor \eqref{eq:scaling}.

Let $\norm{\cdot}$ denote an arbitrary vector norm. The training dataset $\NNtrainingDataIn$ used in \eqref{eq:NN_Training} is employed to define a threshold
\begin{equation}
	\boldsymbol{\delta}_{\text{max}} = \kappa \max_{(k,\SystemState_0,\parameter_0) \in \NNtrainingDataIn} \norm{ \tilde{\boldsymbol{\delta}}(k,\SystemState_0,\parameter_0) }, \qquad \kappa \geq 1,
\end{equation}
where $\kappa$ is a safety factor introduced to account for uncertainty in unseen data.

The quantity \( \norm{\tilde{\boldsymbol{\delta}}(k,\SystemState_0,\parameter_0)} \) then serves as an a posteriori indicator for mitigating the generalization risk of the neural network. Specifically, if
\begin{equation}
	\norm{\tilde{\boldsymbol{\delta}}(k,\SystemState_0,\parameter_0)} \leq \boldsymbol{\delta}_{\text{max}},
	\label{eq:condition}
\end{equation}
the enhanced solver (informed by the network prediction) is employed; otherwise, the algorithm reverts to the classical integrator of order \( p+1 \), ensuring conservative behavior outside the trained regime. The complete strategy for solving an initial value problem using this approach is summarized in Algorithm \ref{alg:safety_net}. These solvers are called \emph{hybrid enhanced solvers} in the sequel.

\begin{algorithm}
	\caption{Enhanced integrators with mitigated generalization risk}
	\label{alg:safety_net}
	\begin{algorithmic}[1]
		\Require Initial state $\SystemState_0$, parameters $\parameter_0$, step size $\stepsize$, tolerance parameters $\mathrm{Atol}_i$, $\mathrm{Rtol}_i$, threshold $\boldsymbol{\delta}_{\text{max}}$, $N$ number of steps, trained $\NN$
		\State $\approximation{\SystemState}_{0}\gets \SystemState_0$
		\For{$k = 0$ to $N-1$}
		\State ${}_p\SystemState_{k+1} \gets {}_p\flowDiscrete( \approximation{\SystemState}_{k})$
		\State ${}_{p+1}\SystemState_{k+1} \gets {}_{p+1}\flowDiscrete( \approximation{\SystemState}_{k})$
		\State ${}_p\boldsymbol{\epsilon}(\approximation{\SystemState}_{k}) \gets {}_{p+1}\SystemState_{k+1} - {}_p\SystemState_{k+1}$
		\State $\boldsymbol{\delta}(k, \approximation{\SystemState}_{k}, \parameter_0) \gets {}_p\boldsymbol{\epsilon}(\approximation{\SystemState}_{k}) - \NN( \approximation{\SystemState}_{k}, \parameter_0)$ 
		\State $\tilde{\boldsymbol{\delta}}(k,\approximation{\SystemState}_{k},\parameter_0) \gets \left(\tfrac{(\boldsymbol{\delta}(k,\approximation{\SystemState}_{k},\parameter_0))_i}{\text{sc}_i}\right)_{i=1}^n$
		\If{$\norm{\tilde{\boldsymbol{\delta}}(k,\approximation{\SystemState}_{k},\parameter_0)} \leq \boldsymbol{\delta}_{\text{max}}$}
		\State $\approximation{\SystemState}_{k+1} \gets \approximation{\SystemState}_{k} + \stepsize^{\integratorOrder+1}\NN( \approximation{\SystemState}_{k}, \parameter_0)$
		\Else
		\State $\approximation{\SystemState}_{k+1} \gets {}_{p+1}\SystemState_{k+1}$
		\EndIf
		\EndFor
		\State \Return $\approximation{\SystemState}_{k}$ for $k\in\{0,\hdots, N\}$
	\end{algorithmic}
\end{algorithm}

\begin{rem}
	Evaluating condition \eqref{eq:condition} for the maximum norm ensures that the discrepancy is bounded component-wise by
	\[
	\abs{( \boldsymbol{\delta}(k,\SystemState_{k},\parameter_0) )_i} \leq \boldsymbol{\delta}_{\text{max}} \cdot \text{sc}_i \quad \text{for all } i=1,\dots,n.
	\]
	This guarantees that the error in each component remains within the scaled tolerance.
	
	For the averaged \(\ell^2\) norm, defined as
	\[
	\norm{\tilde{\boldsymbol{\delta}}}_2 = \left( \frac{1}{n} \sum_{i=1}^n \left( (\tilde{\boldsymbol{\delta}})_i \right)^2 \right)^{1/2},
	\]
	the condition \( \norm{\tilde{\boldsymbol{\delta}}}_2 \leq \boldsymbol{\delta}_{\text{max}} \) implies that the mean squared relative discrepancy satisfies
	\[
	\frac{1}{n} \sum_{i=1}^n \left( \tfrac{( \boldsymbol{\delta}(k,\SystemState_{k},\parameter_0) )_i}{\text{sc}_i} \right)^2 \leq \boldsymbol{\delta}_{\text{max}}^2.
	\]
	This allows for individual components to exceed their scaled tolerance, as long as the overall averaged deviation remains sufficiently small.
\end{rem}

\subsubsection{Efficient implementation of hybrid enhanced solvers}

Embedded Runge--Kutta formulas are a classical tool to obtain multiple approximations of different orders using a shared set of stage evaluations. Such constructions allow for the efficient implementation of higher-order methods by reusing intermediate computations from lower-order schemes. In particular, given a method of order \( p \), an embedded formula of order \( \hat{p} < p \) can often be constructed using the same set of stage values or a subset thereof; see \cite[Section II.4]{hairer2008} and \cite{DormandPrince1980}. These methods can be used to implement the hybrid solvers summarized in Algorithm  \ref{alg:safety_net} by using a minimal number of further vector field evaluations. This is illustrated in the following example.

\begin{exmp}
	\label{exmp:embedded_heun}
	A three-stage third-order Runge--Kutta method can be constructed by combining evaluations at the initial point, a forward Euler step from Example \ref{exmp:Euler}, and a midpoint evaluation using a trapezoidal predictor from Example \ref{exmp:Heun} (see also, e.g., \cite{hairer2008}). The numerical scheme is defined as follows:
	\begin{align}
		k_1(\SystemState) &= f(\SystemState, \parameter_0), \\
		k_2(\SystemState) &= f(\SystemState + \stepsize k_1(\SystemState), \parameter_0), \\
		k_3(\SystemState) &= f\left(\SystemState + \frac{\stepsize}{4}(k_1(\SystemState) + k_2(\SystemState)), \parameter_0\right), \\
		\flowDiscrete( \SystemState) &= \SystemState + \stepsize\left(\frac{1}{6}k_1(\SystemState) + \frac{1}{6}k_2(\SystemState) + \frac{2}{3}k_3(\SystemState)\right).
	\end{align}
	
	For time-invariant systems, i.e.\ if 
    the vector field does not depend explicitly on time, this scheme can be implemented efficiently by reusing computations from a Heun update. Specifically, the evaluations
	\[
	k_1(\SystemState) = f(\SystemState, \parameter_0), \quad k_2(\SystemState) = f(\SystemState + \stepsize k_1(\SystemState), \parameter_0)
	\]
	are already required for Heun's method. To upgrade to the third-order Runge--Kutta method, only one additional evaluation,
	\[
	k_3(\SystemState) = f\left(\SystemState + \tfrac{\stepsize}{4}(k_1(\SystemState) + k_2(\SystemState)), \parameter_0\right),
	\]
	is needed. This makes the higher-order method particularly attractive in contexts where intermediate values from simpler schemes are available or reused across solvers.
\end{exmp}

\section{Numerical results}
\label{sec:numerical_results}
Fatigue analysis is a fundamental part of wind turbine design, required to ensure structural integrity over long operational lifetimes. Due to the stochastic nature of wind conditions, including turbulence, shear, and wake interactions, comprehensive fatigue evaluation requires extensive simulations in a wide range of environmental and operational parameters \cite{Shaler2023,Ding2024,Liu2023}.
These simulations quantify fatigue damage by accounting for cyclic loading variations, which are critical in components such as blades and tower structures. Recent studies highlight the importance of parameter sensitivity \cite{Shaler2023}, probabilistic modeling under coupled wind wave conditions \cite{Ding2024}, and surrogate models to reduce computational cost \cite{Liu2023}. 
In practice, thousands of simulations are often necessary to capture the range of possible loading scenarios, making automated model generation and high-performance computing essential for fatigue-driven design optimization \cite{Schafhirt2022}. Thus, a good accuracy vs efficiency trade-off is crucial.

\subsection{Wind Turbine Model T2B1i1cG and problem statement}

The wind turbine model \texttt{T2B1i1cG}, developed as part of the CADynTurb framework \cite{Geisler2021,Geisler2021a,Geisler2024,Branlard2021}, represents a nonlinear, multi-degree-of-freedom system that captures key structural and electromechanical dynamics of a horizontal-axis wind turbine. The model identifier encodes its configuration: the prefix \texttt{T2} denotes a two-bladed rotor, while \texttt{B1} signifies the inclusion of one primary structural body, typically representing the flexible tower. The tag \texttt{i1} implies one internal degree of freedom associated with drivetrain dynamics, and \texttt{cG} indicates that generator-side dynamics are explicitly modeled.

This model comprises $n=8$ generalized coordinates, each associated with a degree of freedom and its corresponding generalized velocity. Let ${\mathbf{q} \in \mathbb{R}^8}$ denote the vector of generalized coordinates, given by
\begin{equation}
	\mathbf{q} = \begin{pmatrix}
		x_{\mathrm{FA}} & x_{\mathrm{SS}} & \theta_{1,\mathrm{fl}} & \theta_{2,\mathrm{fl}} & \theta_{3,\mathrm{fl}} & \theta_{\mathrm{edg}} & \phi_{\mathrm{rot}} & \phi_{\mathrm{gen}}
	\end{pmatrix}^\top,
\end{equation}
where $x_{\mathrm{FA}}$ and $x_{\mathrm{SS}}$ represent the tower-top displacements in the fore-aft and side-to-side directions, respectively. The variables $\theta_{1,\mathrm{fl}}, \theta_{2,\mathrm{fl}}, \theta_{3,\mathrm{fl}}$ correspond to the flapwise bending modes of the three rotor blades. The quantity $\theta_{\mathrm{edg}}$ captures the collective edgewise bending of the rotor. Finally, $\phi_{\mathrm{rot}}$ denotes the rotor azimuth angle, while $\phi_{\mathrm{gen}}$ represents the azimuthal position of the generator shaft.

\subsubsection{Equations of Motion}

The model equations are derived using Lagrangian mechanics. The dynamics follow from the Euler–Lagrange equations
\begin{equation}
	\frac{d}{dt} \left( \frac{\partial \mathcal{L}}{\partial \dot{q}_i} \right) - \frac{\partial \mathcal{L}}{\partial q_i} = Q_i, \quad i = 1, \dots, 8,
\end{equation}
where the Lagrangian ${\mathcal{L} = T - V}$ is defined as the difference between the total kinetic energy $T$ and the total potential energy $V$ of the system. The terms $Q_i$ represent the generalized non-conservative forces, which include contributions from aerodynamic loads, structural damping, and applied torques (e.g., generator torque).

In compact matrix form, the equations of motion can be expressed as
\begin{equation}
	\mathbf{M}(\mathbf{q}) \ddot{\mathbf{q}} + \mathbf{C}(\mathbf{q}, \dot{\mathbf{q}}) \dot{\mathbf{q}} + \mathbf{K}(\mathbf{q}) = \mathbf{F}_{\mathrm{ext}}(\mathbf{q}, \dot{\mathbf{q}},\parameter_0),
	\label{eq:model_turbine}
\end{equation}
where ${\mathbf{M}(\mathbf{q}) \in \mathbb{R}^{8 \times 8}}$ is the generalized mass matrix, $\mathbf{C}(\mathbf{q}, \dot{\mathbf{q}})$ contains Coriolis and gyroscopic coupling terms, and $\mathbf{K}(\mathbf{q})$ includes conservative forces such as gravitational and elastic restoring forces. The right-hand side vector $\mathbf{F}_{\mathrm{ext}}$ encompasses external inputs and excitations, including aerodynamic forces. Wind parameters are summarized in $\parameter_0\in\setR^{3}$ and contain wind speed, vertical, and horizontal shear. The initial conditions are denoted $\mathbf{q}(0)=\mathbf{q}_0$ and $\dot{\mathbf{q}}(0)=\dot{\mathbf{q}}_0$ for the generalized coordinates and velocities, respectively.

In the following, we are interested in the solutions of \eqref{eq:model_turbine} for $t \in [0, 20] \si{\second}$ under different constant values of the incoming wind speed. The horizontal and vertical wind shear parameters are set to zero for simplicity.  Enhanced solvers with different step size will be considerd

\subsubsection{Wind Turbine Model Parameterization}

To ensure the physical realism and reproducibility of simulation results, all parameters used in the \texttt{T2B1i1cG} model are derived from openly available reference data sets. In particular, the structural and aerodynamic properties correspond to a modified version of the well-established NREL 5-MW reference turbine \cite{Jonkman2009}, which serves as a widely accepted baseline in academic and industrial wind turbine studies.

These parameters include mass and stiffness distributions for the blades and tower, aerodynamic airfoil data, drivetrain inertias, and generator properties. The CADynTurb framework programmatically extracts and utilizes these properties during the symbolic generation of the equations of motion, which promotes consistency and minimizes user-induced variability.

The full list of parameters, along with their numerical values and units, is provided in the associated model directory within the CADynTurb repository. This enables full reproducibility and facilitates the comparison of results across different studies or software environments.

\subsubsection{Model background}

The \texttt{T2B1i1cG} model offers a balance between physical fidelity and computational efficiency. It is particularly well-suited for applications in dynamic simulation, observer design, and model-based control synthesis. Owing to its structured derivation and modular implementation, the model can be exported to code generation frameworks such as \texttt{acados}, or integrated into simulation environments including MATLAB/Simulink and Python-based platforms. Its reduced-order nature makes it ideal for high-speed simulation and real-time estimation tasks while retaining essential structural and electromechanical dynamics. Different models have been considered in \cite{Geisler2024} for the design of extended Kalman filters.
\subsection{Data generation}
The reference data for training the neural network using the optimization problem \eqref{eq:NN_Training} is generated using a seventh-order Runge-Kutta scheme (see, e.g., \cite{hairer2008}) with a constant step size of ${\stepsize_{\text{ref}} = 10^{-4} \si{\second}}$. The initial condition is fixed, with all state components set to zero except for $\phi_{\mathrm{rot}}(0) = \frac{\pi}{3} \si{\radian}$. The solution approximated by this approach is denoted in the following by $\SystemState^{\text{ref.}}$.
Figure \ref{fig:comparison_results_state_evolution} shows exemplarily the time evolution of the state component  $x_{\mathrm{FA}}$  for three different wind speeds.
\begin{figure}[]
	\centering
	\def\svgwidth{0.9\textwidth}
	\graphicspath{{figures/results_far_from_training_data/}}
	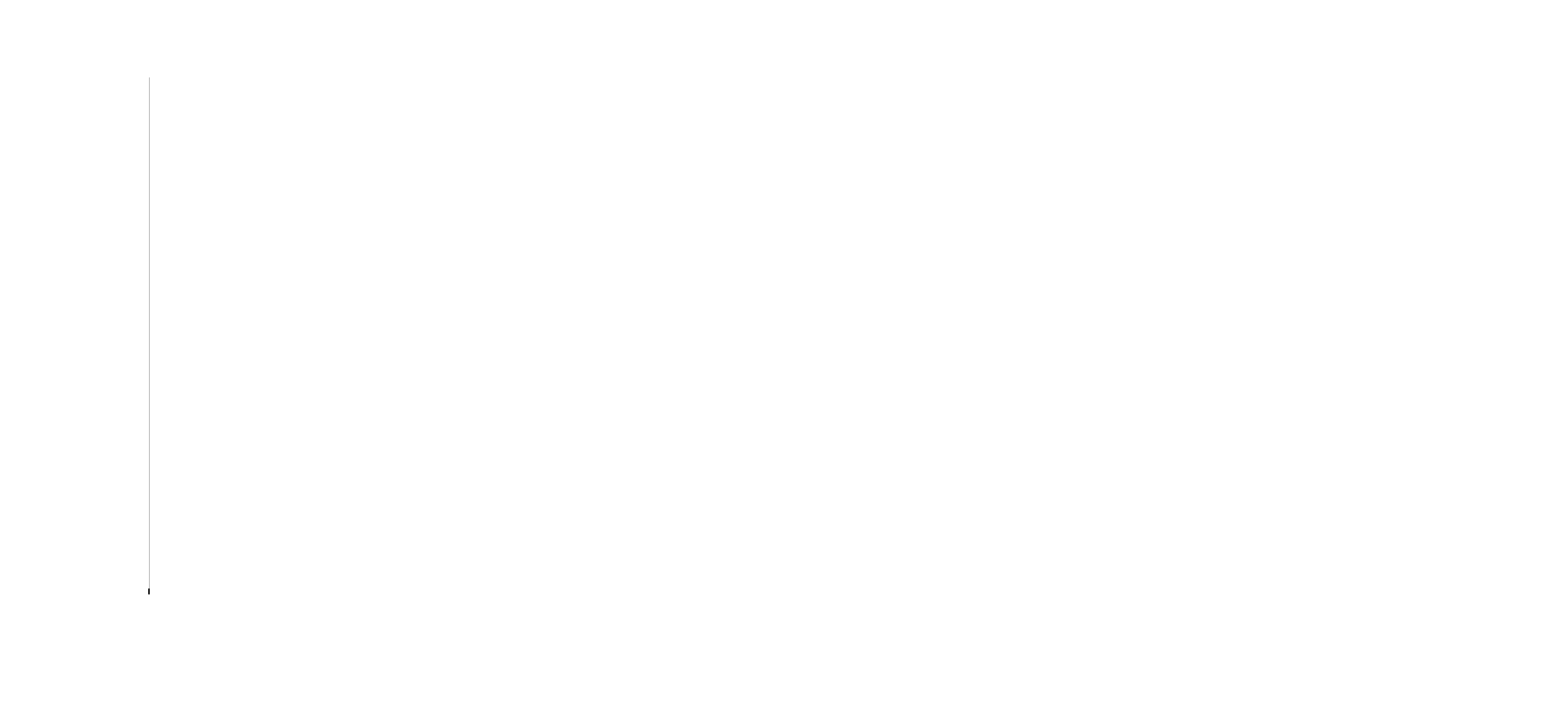
	\caption{Time evolution of the state component $\SystemState_{\mathrm{FA}}$ for three different wind speeds $v_{\text{wind}}$.}
	\label{fig:comparison_results_state_evolution}
\end{figure}
\subsection{Network training}
A neural network $\NN$ with three hidden layers  each containing $111$ neurons, and using the rectified linear unit as the activation function
 is employed,  i.e., 
 \begin{equation}
 	L=3, \quad 
 	\NNactivation(x) = \begin{cases}
 		x, & \text{if } x \geq 0, \\
 		0, & \text{if } x < 0.
 	\end{cases}
 \end{equation}
The input layer has $\NNdimOut=16+1=17$, $16$ entries for the state components and one entry for the wind speed,  and the output is of shape $\NNdimOut=16$. The implementation utilizes PyTorch\footnote{Version 2.5.1 is used.} \cite{NEURIPS2019_9015}   and PyTorch Lightning\footnote{Version 2.5.1 is used.} \cite{Lightning}.

 The learning is performed with six constant wind speeds ${\setParam=\{6.1, 6.7, 7.8, 8.3, 8.8, 9.4\} \si{\meter\per\second}}$. Specifically, \eqref{eq:model_turbine} is solved for these wind parameters, and the resulting approximated trajectories of the state are used as training data.
The full solution trajectories computed using a high-order reference integrator with step size $\stepsize_{\text{ref}}$ are used to construct the training dataset. For notational simplicity, let $\SystemState^{\text{ref.}}_\iterator$ denote the state at time $\iterator \tilde{\stepsize}$, where $\tilde{\stepsize}$ is the step size of the integrator to be enhanced. Then, define the set of time-stamped states as $\setStateTime = \{ \SystemState^{\text{ref.}}_\iterator \mid \iterator \in \mathcal{K}\}$, where $\mathcal{K} = \{0, \dots, N_k\}$ and $N_k\in\setN$ is such that $N_k \tilde{\stepsize} = t_{\text{end}}$.
The corresponding training dataset is given by $\NNtrainingDataIn = \{(\iterator, \SystemState^{\text{ref.}}_\iterator, \parameter) \in \mathcal{K} \times \setStateTime \times \setParam\}$.
The cost function \eqref{eq:NN_loss}, evaluated at a given state $\SystemState^{\text{ref.}}_\iterator$, wind speed $\parameter$, is defined as
\begin{equation}
	\NNloss_{\tilde{\NNparam}}(\iterator, 0, \SystemState^{\text{ref.}}_\iterator, \parameter)
	= \left\lVert \NN( \SystemState^{\text{ref.}}_\iterator, \parameter)
	- \tfrac{1}{\tilde{\stepsize}^{\integratorOrder+1}}\left(\SystemState^{\text{ref.}}_{\iterator+1} - \flowDiscrete[\tilde{\stepsize}]( \SystemState^{\text{ref.}}_\iterator)\right) \right\rVert^2.
	\label{eq:NN_loss_wind}
\end{equation}

The network is trained using mini-batch gradient descent, where the training data is divided into batches of fixed size to improve memory efficiency and training stability. Batch processing enables faster convergence and smoother optimization dynamics, especially in the presence of noisy gradients. The AdamW optimizer \cite{Loshchilov2019} is used to solve \eqref{eq:NN_OP}.

Data generated using the wind speeds $\{5, 7.2\} \si{\meter\per\second}$ are used for validating the prediction of the local error and tuning the \texttt{ReduceLROnPlateau} method in PyTorch Lightning. This scheduler reduces the learning rate when the loss function \eqref{eq:NN_loss} evaluated on the validation dataset ceases to improve. If no improvement is observed for a specified number of epochs, the learning rate is reduced by a predefined factor. This mechanism facilitates finer control over the learning process and helps to mitigate overfitting.
The hyperparameter optimization framework Optuna \cite{Akiba2019} is employed to identify optimal values for the batch size, learning rate, and scheduler parameters, thereby enhancing training efficiency and predictive performance. The results discussed in the next chapter are derived using the models trained using the best hyperparameters.

To ensure that the correction learned by the neural network is applied consistently across all components of the system state, both the training and validation data are scaled independently. Specifically, each state component is linearly mapped from its minimum and maximum values to the interval $[0,1]$, ensuring a uniform input and output range across all components. 

To preserve this scaling during inference and to integrate it seamlessly into the network architecture, two additional fixed linear layers are appended—one at the input and one at the output of the network. These layers apply the inverse and forward scaling transformations, respectively, allowing the neural network to operate entirely in the normalized space while maintaining compatibility with the original physical scale of the data. The loss function \eqref{eq:NN_loss} is evaluated on the scaled data, promoting numerical stability and consistent training dynamics across all state dimensions.

\subsection{Comparison of numerical schemes}
In the following, let \( n_{\textsc{RK}} \) and \( n_{\NN} \) denote the number of evaluations of the vector field \( f \) from \eqref{eq:ODE}, and the number of evaluations of a neural network \( \NN \), respectively, that are required by a numerical integration scheme. For instance, the classical Heun method requires \( n_{\textsc{RK}} = 2 \) vector field evaluations and \( n_{\NN} = 0 \) network evaluations. To assess the efficiency and accuracy of different integration approaches, both the total simulation time and the global error of the resulting trajectories are compared.

To quantify the computational effort to approximate a solution, we introduce the performance indicator
\begin{equation}
	\delta_t = \frac{n_{\textsc{RK}}\, \textsc{time}(f) + n_{\NN}\, \textsc{time}(\NN)}{\stepsize},
	\label{eq:timing}
\end{equation}
where \( \textsc{time}(\kappa) \) denotes the  time required by the processor to evaluate a given mapping \( \kappa:x\mapsto\kappa(x) \) at an arbitrary input \( x \). The quantity \( \delta_t \) thus approximates the total computational effort per unit of simulation time and serves as a proxy for the time efficiency of the numerical method under consideration.

To evaluate the accuracy of each method, let \( \SystemState^{\text{ref.}}(k\stepsize) \) denote a reference solution of \eqref{eq:model_turbine}, and let \( \tilde{\SystemState}_k \) represent the corresponding approximation obtained by the method under analysis. Denoting by \( (\SystemState^{\text{ref.}})_{k,i} \) the value of the \( i \)-th component of the reference solution at time \( k\stepsize \), the global relative squared error is defined as
\begin{equation}
	\delta_e = \frac{1}{nN} \sum_{k=1}^N \sum_{i=1}^n \left( \frac{|(\SystemState^{\text{ref.}})_{k,i} - (\tilde{\SystemState})_{k,i}|}{|(\SystemState^{\text{ref.}})_{k,i}| + \varepsilon} \right)^2,
	\label{eq:global_error}
\end{equation}
with \( \varepsilon = 10^{-18} \) to prevent division by zero, \( N \) the number of integration steps, and \( n=16\) the dimension of the state vector. The evaluation interval for all methods is fixed to \( [0, 20]\,\si{\second} \).

\subsection{Results and discussion}
For the testing of the performance of the enhanced integrators proposed in Section~\ref{sec:enhanced_integrators} 200 simulations with different speeds in the interval $[5,10]\si{\meter\per\second}$ are considered. The Heun integrator discussed in Example~\ref{exmp:Heun}
shall be enhanced for different step sizes between $\{0.01,...,0.0002\}\si{\second}$ and applied in its hybrid enhanced version as given in Algorithm~\ref{alg:safety_net}.
The results are then compared to the Runge-Kutta method of order three given in Example~\ref{exmp:embedded_heun} (denoted in the following by RK3) and the Heun method with Richarson extrapolation as described in Section~\ref{sec:enhanced_integrators} using the metrics \eqref{eq:timing} and \eqref{eq:global_error}.

Figure~\ref{fig:comparison_results_median_Error_vs_time} shows the evolution of the median value of the error defined in~\eqref{eq:global_error} with respect to the time metric introduced in~\eqref{eq:timing}. It is evident that the proposed enhanced integrators outperform the classical ones in terms of the speed–accuracy trade-off. This becomes particularly clear when comparing results at a fixed error level, i.e., along horizontal lines.

Figure~\ref{fig:comparison_results_median_Error_vs_time} shows the evolution of the median value of the error defined in~\eqref{eq:global_error} with respect to the time metric introduced in~\eqref{eq:timing}. It is evident that the proposed enhanced integrators outperform the classical ones in terms of the speed–accuracy trade-off. This becomes particularly clear when comparing results at a fixed error level, i.e., along horizontal lines. The evolution of the error with respect to \eqref{eq:timing}, the time required for the evaluation of the vector fields and neural networks, suggest that the leading term $\delta$ in the local error \eqref{eq:leading_term_local_error} can be well approximated for step sizes that are greater or equal to $\SI{0.3}{\milli\second}$.

A comparison between the enhanced Heun method and the hybrid solver indicates that both can achieve comparable accuracy; however, the hybrid method is slower. This is attributed to the additional evaluation of the vector field required for the second local error estimate, as described in Algorithm~\ref{alg:safety_net}.

The boxplots in Figures~\ref{fig:comparison_results_box_plot1} and~\ref{fig:comparison_results_box_plot2} display the distribution of the metrics~\eqref{eq:timing} and~\eqref{eq:global_error} along the horizontal and vertical axes, respectively. For comparable values of the time metric~\eqref{eq:timing}, the enhanced Heun method consistently yields significantly lower errors than the RK3 scheme as measured by~\eqref{eq:global_error}. The hybrid solver performs  incurs higher computational cost due to the additional evaluations.  Nonetheless, it exhibits fewer outliers, suggesting a more consistent performance across varying conditions.

\begin{figure}[]
	\centering
	\def\svgwidth{\textwidth}
	\graphicspath{{figures/results_in_training_area/}}
	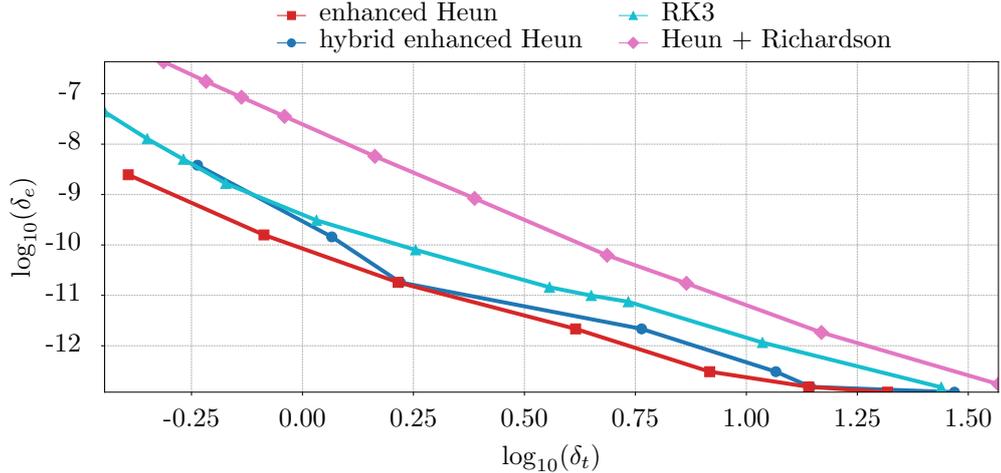
	\caption{Comparison of four solvers with respect to computational effort \( \delta_t \) from \eqref{eq:timing} and normalized global error \( \delta_e \) from \eqref{eq:global_error}. Each point shows the median \( \delta_t \) over \( 10^6 \) evaluations and \( \delta_e \) over 200 simulations with different speeds in the interval $[5,10]\si{\meter\per\second}$.
	} 
	\label{fig:comparison_results_median_Error_vs_time}
\end{figure}

	\begin{figure}
	\def\svgwidth{\textwidth}
	\graphicspath{{figures/results_in_training_area/}}
	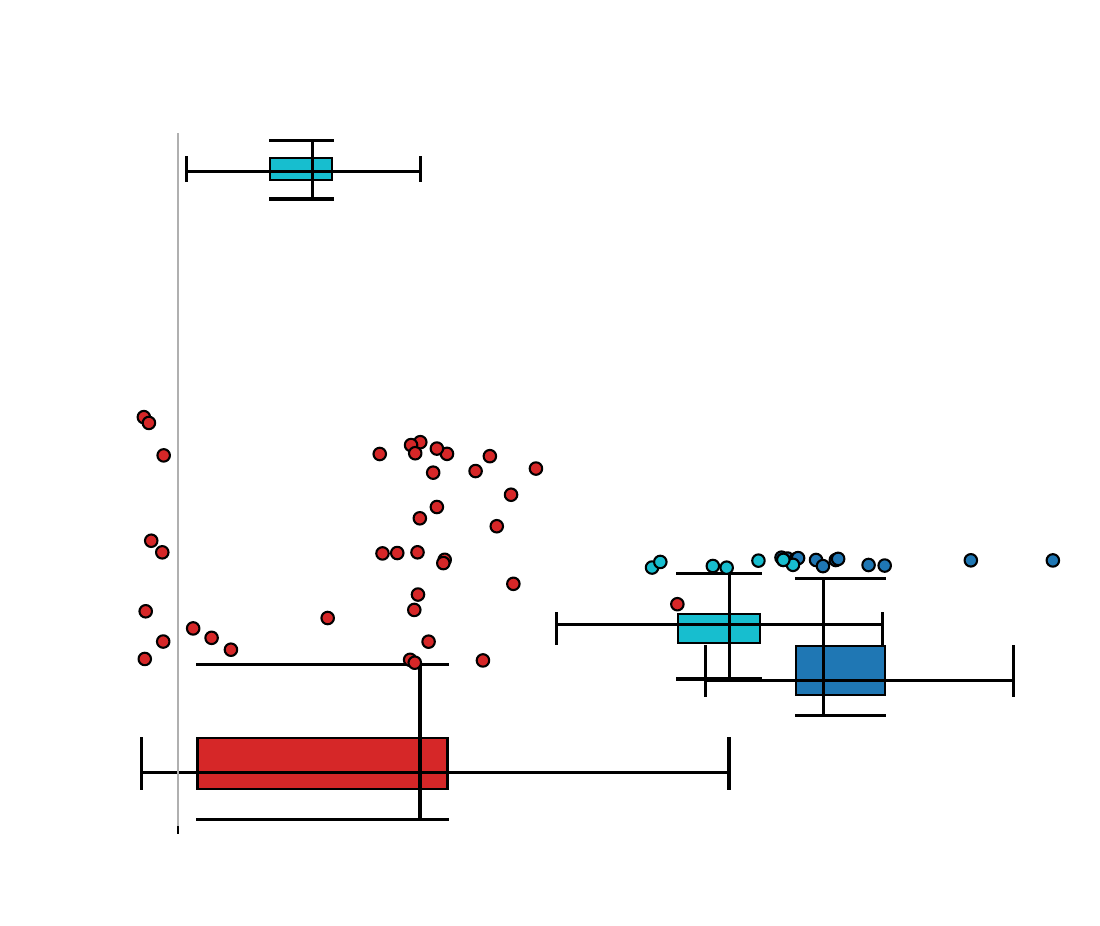
	\caption{Normalized global error \( \delta_e \), as defined in \eqref{eq:global_error}, computed over  200 simulations with different speeds in the interval $[5,10]\si{\meter\per\second}$.  The time step for the neural network based solver is \( h = \SI{10}{\milli\second}\). The corresponding computational cost \( \delta_t \), defined in \eqref{eq:timing}, is evaluated over \( 10^6 \) function calls. The isolated points correspond to outliers. The two RK3 schemes have step sizes of $ \SI{15}{\milli\second}$ and $ \SI{10}{\milli\second}$.
	} 
	\label{fig:comparison_results_box_plot1}
\end{figure}

\begin{figure}
	\def\svgwidth{\textwidth}
	\graphicspath{{figures/results_in_training_area/}}
	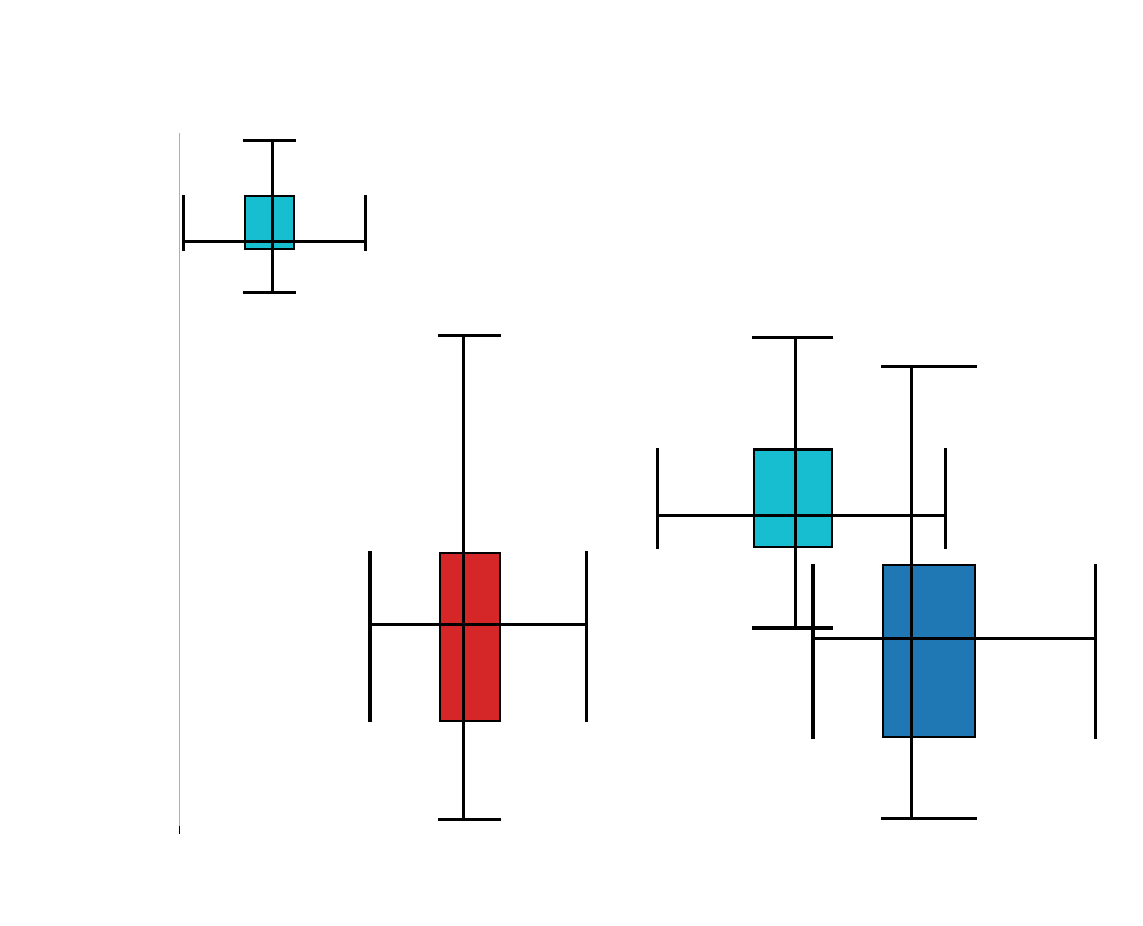
	\caption{Normalized global error \( \delta_e \), as defined in \eqref{eq:global_error}, computed over  200 simulations with different speeds in the interval $[5,10]\si{\meter\per\second}$. The time step for the neural network based solver is \( h = \SI{5}{\milli\second}\). The corresponding computational cost \( \delta_t \), defined in \eqref{eq:timing}, is evaluated over \( 10^6 \) function calls. The two RK3 schemes have step sizes of $ \SI{15}{\milli\second}$ and $ \SI{10}{\milli\second}$
	} 
	\label{fig:comparison_results_box_plot2}
\end{figure}

To further analyze the advantages of the hybrid enhanced solver summarized in Algorithm~\ref{alg:safety_net}, an additional set of $200$ test simulations is considered, with wind speeds sampled from the interval $[5, 20]~\si{\meter\per\second}$. The sampling is such that $75\%$ of the wind speeds lie within this interval, i.e., close to the range used during the training of the neural networks.
The time evolution of the state component $\SystemState_{\mathrm{FA}}$ in Figure~\ref{fig:comparison_results_state_evolution} shows that variations in wind speed can significantly alter the system dynamics. Therefore, it is important to assess the performance of the enhanced solvers outside the wind speed range used in training.

The results are presented in the boxplots in Figures~\ref{fig:comparison_results_box_plot1_far} and~\ref{fig:comparison_results_box_plot2_far}. It can be observed that the error associated with the enhanced Heun method increases significantly in this setting, with several outliers exceeding the maximum error attained by the third-order Runge–Kutta method. In contrast, the performance of the hybrid approach remains good, demonstrating robustness even outside the training regime.

\begin{figure}
	\def\svgwidth{\textwidth}
	\graphicspath{{figures/results_far_from_training_data/}}
	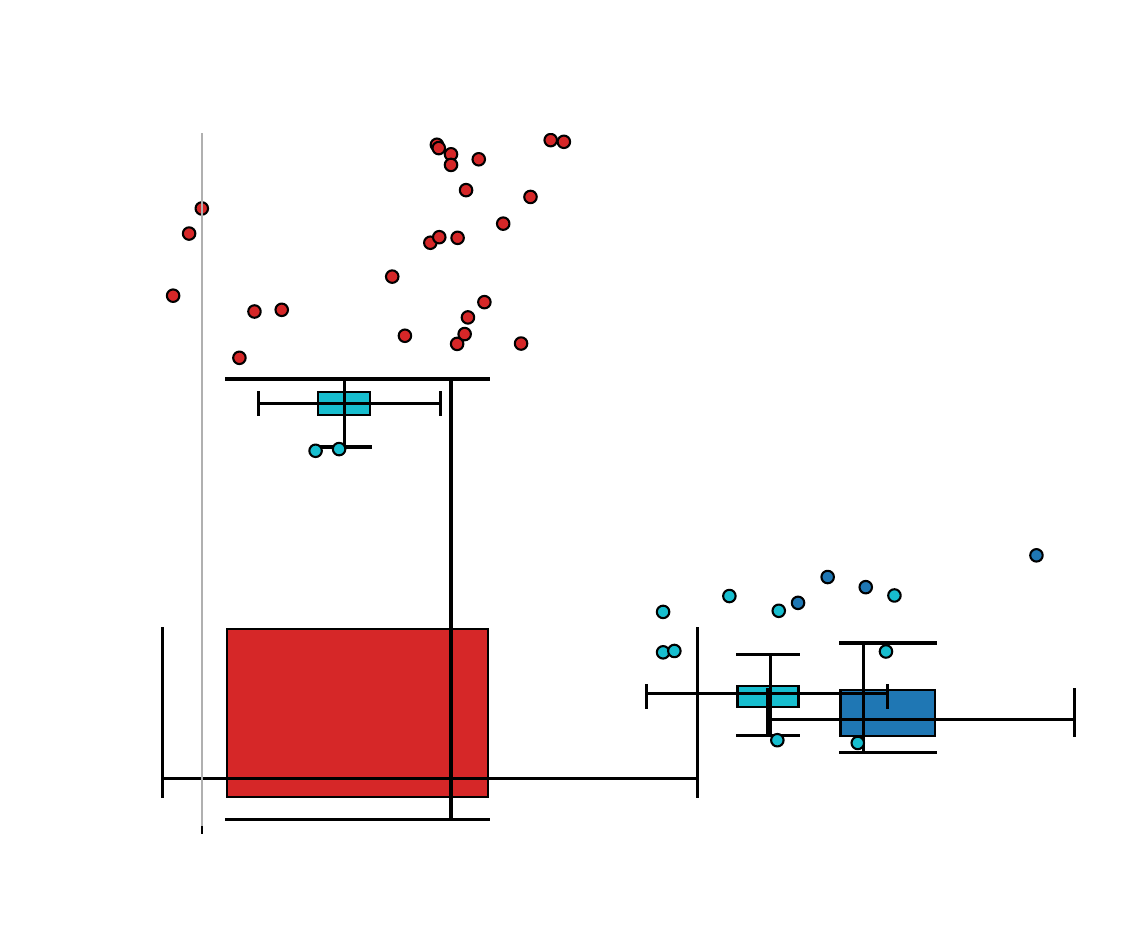
	\caption{Normalized global error \( \delta_e \), as defined in \eqref{eq:global_error}, computed over 200 simulations with varying wind speeds in the interval $[5,10]\si{\meter\per\second}$.  The values are such that $75\%$ of them are in the interval $[5,20]\si{\meter\per\second}$.  The time step for the neural network based solver is \( h = \SI{10}{\milli\second}\). The corresponding computational cost \( \delta_t \), defined in \eqref{eq:timing}, is evaluated over \( 10^6 \) function calls. The isolated points correspond to outliers. The two RK3 schemes have step sizes of $ \SI{15}{\milli\second}$ and $ \SI{10}{\milli\second}$
	} 
	\label{fig:comparison_results_box_plot1_far}
\end{figure}
\begin{figure}
	\def\svgwidth{\textwidth}
	\graphicspath{{figures/results_far_from_training_data/}}
	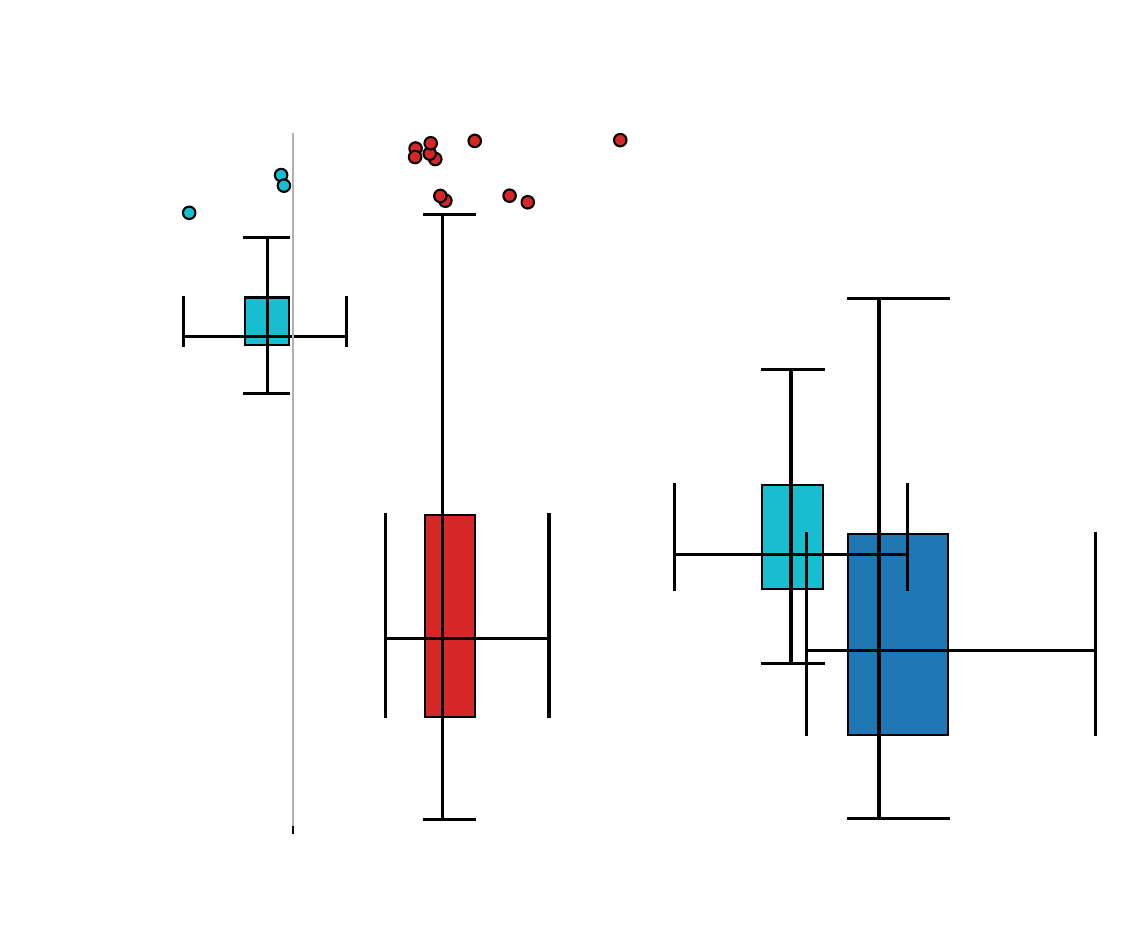
	\caption{Normalized global error \( \delta_e \), as defined in \eqref{eq:global_error}, computed over 200 simulations with varying wind speeds in the interval $[5,10]\si{\meter\per\second}$.  The values are such that $75\%$ are them in the interval $[5,20]\si{\meter\per\second}$.  The time step for the neural network based solver is \( h = \SI{5}{\milli\second}\). The corresponding computational cost \( \delta_t \), defined in \eqref{eq:timing}, is evaluated over \( 10^6 \) function calls. The isolated points correspond to outliers. The two RK3 schemes have step sizes of $ \SI{15}{\milli\second}$ and $ \SI{10}{\milli\second}$
	} 
	\label{fig:comparison_results_box_plot2_far}
\end{figure}
\section{Conclusion} \label{sec:Conclusion}
In this work, we have introduced a class of neural network-enhanced integrators that leverage data-driven error correction to improve the accuracy of classical numerical methods for solving differential equations. By training neural networks to let them approximate the local truncation error, these enhanced integrators incorporate additive correction terms into standard schemes, thereby reducing local errors and preserving the stability and convergence properties of the underlying integrators.

Our analytical investigation via backward error analysis confirms that the proposed methodology maintains the theoretical guarantees of classical Runge–Kutta schemes while effectively mitigating the error propagation inherent in long-term simulations. Furthermore, the incorporation of embedded Runge–Kutta techniques serves to control the generalization risk, ensuring that the neural network corrections remain robust even when the state trajectory ventures into regions not encountered during training.

Extensive numerical studies, including realistic simulations of wind turbine dynamics based on models parameterized with data from OpenFast, demonstrate that the enhanced integrators not only achieve the desired accuracy but also offer significant improvements in computational efficiency. These results underscore the potential of combining classical numerical analysis with modern machine learning approaches to address complex, high-dimensional dynamical systems.

Future work will focus on further refining the integration of neural network corrections with adaptive error control strategies and extending the framework to a broader class of integrators. Moreover, problems specific cost functions for the training of neural networks shall be derived and the use of the enhanced integrators for load prediction of life-cycle simulations of wind turbines will be considered. The training of the neural network when the entire state is not known but only a function of it deserves future research.

\paragraph*{Acknowledgment}
The authors thank Jens Geisler and Markus Herrmann-Wicklmayr for most helpful support during the use of the CADynTurb framework \cite{Geisler2021} and Balázs Kovács for the discussions on Section~\ref{sec:safety_net}.

\bibliography{bib_file}
\end{document}

%% file: macros_text.tex
\newcommand{\ie}{i.e.}

%% file: macros_math.tex
\newcommand{\stateSymbol}{x}
\newcommand{\paramSymbol}{p}
\newcommand{\SystemState}{\boldsymbol{\stateSymbol}}
\newcommand{\parameter}{\boldsymbol{\paramSymbol}}

\newcommand{\setParam}{\mathcal{\MakeUppercase\paramSymbol}}

\newcommand{\setStateTime}{\mathcal{S}_{x}}
\newcommand{\dimensionState}{n_\stateSymbol}
\newcommand{\dimensionParam}{n_\paramSymbol}
\newcommand{\approximation}[1]{\tilde{#1}}

\newcommand{\flowCont}[1][t]{\nu_{#1}}
\newcommand{\flowDiscrete}[1][\stepsize]{\boldsymbol{\eta}_{#1}}

\newcommand{\iterator}{k}
\newcommand{\stepsize}{h}

\newcommand{\integratorOrder}{p}
\newcommand{\integratorError}{\boldsymbol{e}}

\newcommand{\NN}{\mathcal{N}_\NNparam}
\newcommand{\NNactivation}{\xi}
\newcommand{\NNactivationStuck}{\Xi}
\newcommand{\NNdepth}{L}
\newcommand{\NNparam}{\theta}
\newcommand{\NNsetParam}{\Theta}
\newcommand{\NNloss}{\mathcal{L}}
\newcommand{\NNtrainingDataIn}{\mathcal{D}}
\newcommand{\NNtrainingDataOut}{\mathcal{E}}
\newcommand{\NNcarTrainingData}{N_{\NNtrainingDataIn}}
\newcommand{\NNdimIn}{n_{\mathrm{in}}}
\newcommand{\NNdimOut}{n_{\mathrm{out}}}
\newcommand{\NNerror}{\epsilon_{\NN}}

\newcommand{\identityMarix}[1][]{I_{}}

\DeclareMathOperator*{\argmin}{arg\,min}
\newcommand{\norm}[1]{\left\lVert#1\right\rVert}
\newcommand{\abs}[1]{\left|#1\right|}
\newcommand{\setR}{\mathbb{R}}
\newcommand{\setN}{\mathbb{N}}

\newcommand{\auxCalc}[1]{%
	\ifthenelse{\boolean{showcalculations}}{
		\gdef\mycolor{blue}#1 \gdef\mycolor{black}\color{black}}{}}
\everymath\expandafter{\the\everymath\color{\mycolor}}  
\def\mycolor{black}
\newcommand{\myVector}[1]{(#1)}

%% file: macros_todo.tex
\usepackage[colorinlistoftodos,textsize=tiny]{todonotes}
\definecolor{kathrin}{rgb}{0.5,0.0,0.5} 
\definecolor{kathrinCom}{rgb}{0.4,0.1,0.4} 
\definecolor{kathrinCor}{rgb}{0.0,0.9,0.9} 
\definecolor{amine}{rgb}{0.0,0.75,0.0}

\newcommand{\kCh}[2]{
	\ifthenelse{\boolean{showchanges}}{ \sout{#1}
		{\color{kathrinCor}{#2}\;}  \addcontentsline{tdo}{todo}{{\color{kathrinCor}{K. change }}}}{#2{}}
	}

%% file: macros_thm.tex
\newtheorem{exmp}{Example}
\newtheorem{prop}{Proposition}
\theoremstyle{remark}

\newtheorem{rem}{Remark}

%% file: figures/results_far_from_training_data/comparison_results_state_evolution_far.pdf_tex
\begingroup%
  \makeatletter%
  \providecommand\color[2][]{%
    \errmessage{(Inkscape) Color is used for the text in Inkscape, but the package 'color.sty' is not loaded}%
    \renewcommand\color[2][]{}%
  }%
  \providecommand\transparent[1]{%
    \errmessage{(Inkscape) Transparency is used (non-zero) for the text in Inkscape, but the package 'transparent.sty' is not loaded}%
    \renewcommand\transparent[1]{}%
  }%
  \providecommand\rotatebox[2]{#2}%
  \newcommand*\fsize{\dimexpr\f@size pt\relax}%
  \newcommand*\lineheight[1]{\fontsize{\fsize}{#1\fsize}\selectfont}%
  \ifx\svgwidth\undefined%
    \setlength{\unitlength}{1020.29248047bp}%
    \ifx\svgscale\undefined%
      \relax%
    \else%
      \setlength{\unitlength}{\unitlength * \real{\svgscale}}%
    \fi%
  \else%
    \setlength{\unitlength}{\svgwidth}%
  \fi%
  \global\let\svgwidth\undefined%
  \global\let\svgscale\undefined%
  \makeatother%
  \begin{picture}(1,0.46054715)%
    \lineheight{1}%
    \setlength\tabcolsep{0pt}%
    \put(0,0){\includegraphics[width=\unitlength,page=1]{comparison_results_state_evolution_far.pdf}}%
    \put(0.09514428,0.0499979){\makebox(0,0)[t]{\lineheight{1.25}\smash{\begin{tabular}[t]{c}0.0\end{tabular}}}}%
    \put(0,0){\includegraphics[width=\unitlength,page=2]{comparison_results_state_evolution_far.pdf}}%
    \put(0.20452469,0.0499979){\makebox(0,0)[t]{\lineheight{1.25}\smash{\begin{tabular}[t]{c}2.5\end{tabular}}}}%
    \put(0,0){\includegraphics[width=\unitlength,page=3]{comparison_results_state_evolution_far.pdf}}%
    \put(0.31390508,0.0499979){\makebox(0,0)[t]{\lineheight{1.25}\smash{\begin{tabular}[t]{c}5.0\end{tabular}}}}%
    \put(0,0){\includegraphics[width=\unitlength,page=4]{comparison_results_state_evolution_far.pdf}}%
    \put(0.42328549,0.0499979){\makebox(0,0)[t]{\lineheight{1.25}\smash{\begin{tabular}[t]{c}7.5\end{tabular}}}}%
    \put(0,0){\includegraphics[width=\unitlength,page=5]{comparison_results_state_evolution_far.pdf}}%
    \put(0.53266586,0.0499979){\makebox(0,0)[t]{\lineheight{1.25}\smash{\begin{tabular}[t]{c}10.0\end{tabular}}}}%
    \put(0,0){\includegraphics[width=\unitlength,page=6]{comparison_results_state_evolution_far.pdf}}%
    \put(0.6420463,0.0499979){\makebox(0,0)[t]{\lineheight{1.25}\smash{\begin{tabular}[t]{c}12.5\end{tabular}}}}%
    \put(0,0){\includegraphics[width=\unitlength,page=7]{comparison_results_state_evolution_far.pdf}}%
    \put(0.75142668,0.0499979){\makebox(0,0)[t]{\lineheight{1.25}\smash{\begin{tabular}[t]{c}15.0\end{tabular}}}}%
    \put(0,0){\includegraphics[width=\unitlength,page=8]{comparison_results_state_evolution_far.pdf}}%
    \put(0.86080711,0.0499979){\makebox(0,0)[t]{\lineheight{1.25}\smash{\begin{tabular}[t]{c}17.5\end{tabular}}}}%
    \put(0,0){\includegraphics[width=\unitlength,page=9]{comparison_results_state_evolution_far.pdf}}%
    \put(0.97018749,0.0499979){\makebox(0,0)[t]{\lineheight{1.25}\smash{\begin{tabular}[t]{c}20.0\end{tabular}}}}%
    \put(0.53266586,0.01031812){\makebox(0,0)[t]{\lineheight{1.25}\smash{\begin{tabular}[t]{c}$t$ in s\end{tabular}}}}%
    \put(0,0){\includegraphics[width=\unitlength,page=10]{comparison_results_state_evolution_far.pdf}}%
    \put(0.07211167,0.0792326){\makebox(0,0)[rt]{\lineheight{1.25}\smash{\begin{tabular}[t]{r}0.0\end{tabular}}}}%
    \put(0,0){\includegraphics[width=\unitlength,page=11]{comparison_results_state_evolution_far.pdf}}%
    \put(0.07211167,0.13348971){\makebox(0,0)[rt]{\lineheight{1.25}\smash{\begin{tabular}[t]{r}0.1\end{tabular}}}}%
    \put(0,0){\includegraphics[width=\unitlength,page=12]{comparison_results_state_evolution_far.pdf}}%
    \put(0.07211167,0.18774682){\makebox(0,0)[rt]{\lineheight{1.25}\smash{\begin{tabular}[t]{r}0.2\end{tabular}}}}%
    \put(0,0){\includegraphics[width=\unitlength,page=13]{comparison_results_state_evolution_far.pdf}}%
    \put(0.07211167,0.24200396){\makebox(0,0)[rt]{\lineheight{1.25}\smash{\begin{tabular}[t]{r}0.3\end{tabular}}}}%
    \put(0,0){\includegraphics[width=\unitlength,page=14]{comparison_results_state_evolution_far.pdf}}%
    \put(0.07211167,0.29626108){\makebox(0,0)[rt]{\lineheight{1.25}\smash{\begin{tabular}[t]{r}0.4\end{tabular}}}}%
    \put(0,0){\includegraphics[width=\unitlength,page=15]{comparison_results_state_evolution_far.pdf}}%
    \put(0.07211167,0.3505182){\makebox(0,0)[rt]{\lineheight{1.25}\smash{\begin{tabular}[t]{r}0.5\end{tabular}}}}%
    \put(0,0){\includegraphics[width=\unitlength,page=16]{comparison_results_state_evolution_far.pdf}}%
    \put(0.07211167,0.40477533){\makebox(0,0)[rt]{\lineheight{1.25}\smash{\begin{tabular}[t]{r}0.6\end{tabular}}}}%
    \put(0.0189725,0.2479583){\rotatebox{90}{\makebox(0,0)[t]{\lineheight{1.25}\smash{\begin{tabular}[t]{c}$\SystemState_{\mathrm{FA}}$in m\end{tabular}}}}}%
    \put(0,0){\includegraphics[width=\unitlength,page=17]{comparison_results_state_evolution_far.pdf}}%
    \put(0.12570661,0.43530194){\makebox(0,0)[lt]{\lineheight{1.25}\smash{\begin{tabular}[t]{l}$v_{\text{wind}}$= 5 in m/s\end{tabular}}}}%
    \put(0,0){\includegraphics[width=\unitlength,page=18]{comparison_results_state_evolution_far.pdf}}%
    \put(0.43550502,0.43530194){\makebox(0,0)[lt]{\lineheight{1.25}\smash{\begin{tabular}[t]{l}$v_{\text{wind}}$= 6 in m/s\end{tabular}}}}%
    \put(0,0){\includegraphics[width=\unitlength,page=19]{comparison_results_state_evolution_far.pdf}}%
    \put(0.74530343,0.43530194){\makebox(0,0)[lt]{\lineheight{1.25}\smash{\begin{tabular}[t]{l}$v_{\text{wind}}$= 11 in m/s\end{tabular}}}}%
  \end{picture}%
\endgroup%

%% file: figures/results_in_training_area/comparison_results_median_Error_vs_time.pdf_tex
\begingroup%
  \makeatletter%
  \providecommand\color[2][]{%
    \errmessage{(Inkscape) Color is used for the text in Inkscape, but the package 'color.sty' is not loaded}%
    \renewcommand\color[2][]{}%
  }%
  \providecommand\transparent[1]{%
    \errmessage{(Inkscape) Transparency is used (non-zero) for the text in Inkscape, but the package 'transparent.sty' is not loaded}%
    \renewcommand\transparent[1]{}%
  }%
  \providecommand\rotatebox[2]{#2}%
  \newcommand*\fsize{\dimexpr\f@size pt\relax}%
  \newcommand*\lineheight[1]{\fontsize{\fsize}{#1\fsize}\selectfont}%
  \ifx\svgwidth\undefined%
    \setlength{\unitlength}{992.50469971bp}%
    \ifx\svgscale\undefined%
      \relax%
    \else%
      \setlength{\unitlength}{\unitlength * \real{\svgscale}}%
    \fi%
  \else%
    \setlength{\unitlength}{\svgwidth}%
  \fi%
  \global\let\svgwidth\undefined%
  \global\let\svgscale\undefined%
  \makeatother%
  \begin{picture}(1,0.48445929)%
    \lineheight{1}%
    \setlength\tabcolsep{0pt}%
    \put(0,0){\includegraphics[width=\unitlength,page=1]{comparison_results_median_Error_vs_time.pdf}}%
    \put(0.18059871,0.04942075){\makebox(0,0)[t]{\lineheight{1.25}\smash{\begin{tabular}[t]{c}-0.25\end{tabular}}}}%
    \put(0,0){\includegraphics[width=\unitlength,page=2]{comparison_results_median_Error_vs_time.pdf}}%
    \put(0.29228214,0.04942075){\makebox(0,0)[t]{\lineheight{1.25}\smash{\begin{tabular}[t]{c}0.00\end{tabular}}}}%
    \put(0,0){\includegraphics[width=\unitlength,page=3]{comparison_results_median_Error_vs_time.pdf}}%
    \put(0.40396556,0.04942075){\makebox(0,0)[t]{\lineheight{1.25}\smash{\begin{tabular}[t]{c}0.25\end{tabular}}}}%
    \put(0,0){\includegraphics[width=\unitlength,page=4]{comparison_results_median_Error_vs_time.pdf}}%
    \put(0.51564897,0.04942075){\makebox(0,0)[t]{\lineheight{1.25}\smash{\begin{tabular}[t]{c}0.50\end{tabular}}}}%
    \put(0,0){\includegraphics[width=\unitlength,page=5]{comparison_results_median_Error_vs_time.pdf}}%
    \put(0.62733242,0.04942075){\makebox(0,0)[t]{\lineheight{1.25}\smash{\begin{tabular}[t]{c}0.75\end{tabular}}}}%
    \put(0,0){\includegraphics[width=\unitlength,page=6]{comparison_results_median_Error_vs_time.pdf}}%
    \put(0.73901583,0.04942075){\makebox(0,0)[t]{\lineheight{1.25}\smash{\begin{tabular}[t]{c}1.00\end{tabular}}}}%
    \put(0,0){\includegraphics[width=\unitlength,page=7]{comparison_results_median_Error_vs_time.pdf}}%
    \put(0.85069924,0.04942075){\makebox(0,0)[t]{\lineheight{1.25}\smash{\begin{tabular}[t]{c}1.25\end{tabular}}}}%
    \put(0,0){\includegraphics[width=\unitlength,page=8]{comparison_results_median_Error_vs_time.pdf}}%
    \put(0.96238266,0.04942075){\makebox(0,0)[t]{\lineheight{1.25}\smash{\begin{tabular}[t]{c}1.50\end{tabular}}}}%
    \put(0.54297442,0.01058025){\makebox(0,0)[t]{\lineheight{1.25}\smash{\begin{tabular}[t]{c}$\log_{10}(\delta_t)$\end{tabular}}}}%
    \put(0,0){\includegraphics[width=\unitlength,page=9]{comparison_results_median_Error_vs_time.pdf}}%
    \put(0.0695258,0.12552581){\makebox(0,0)[rt]{\lineheight{1.25}\smash{\begin{tabular}[t]{r}-12\end{tabular}}}}%
    \put(0,0){\includegraphics[width=\unitlength,page=10]{comparison_results_median_Error_vs_time.pdf}}%
    \put(0.0695258,0.17663537){\makebox(0,0)[rt]{\lineheight{1.25}\smash{\begin{tabular}[t]{r}-11\end{tabular}}}}%
    \put(0,0){\includegraphics[width=\unitlength,page=11]{comparison_results_median_Error_vs_time.pdf}}%
    \put(0.0695258,0.22774492){\makebox(0,0)[rt]{\lineheight{1.25}\smash{\begin{tabular}[t]{r}-10\end{tabular}}}}%
    \put(0,0){\includegraphics[width=\unitlength,page=12]{comparison_results_median_Error_vs_time.pdf}}%
    \put(0.0695258,0.27885447){\makebox(0,0)[rt]{\lineheight{1.25}\smash{\begin{tabular}[t]{r}-9\end{tabular}}}}%
    \put(0,0){\includegraphics[width=\unitlength,page=13]{comparison_results_median_Error_vs_time.pdf}}%
    \put(0.0695258,0.32996402){\makebox(0,0)[rt]{\lineheight{1.25}\smash{\begin{tabular}[t]{r}-8\end{tabular}}}}%
    \put(0,0){\includegraphics[width=\unitlength,page=14]{comparison_results_median_Error_vs_time.pdf}}%
    \put(0.0695258,0.38107358){\makebox(0,0)[rt]{\lineheight{1.25}\smash{\begin{tabular}[t]{r}-7\end{tabular}}}}%
    \put(0.01797253,0.25139239){\rotatebox{90}{\makebox(0,0)[t]{\lineheight{1.25}\smash{\begin{tabular}[t]{c}$\log_{10}(\delta_e)$\end{tabular}}}}}%
    \put(0,0){\includegraphics[width=\unitlength,page=15]{comparison_results_median_Error_vs_time.pdf}}%
    \put(0.30942802,0.46084448){\makebox(0,0)[lt]{\lineheight{1.25}\smash{\begin{tabular}[t]{l}enhanced Heun\end{tabular}}}}%
    \put(0,0){\includegraphics[width=\unitlength,page=16]{comparison_results_median_Error_vs_time.pdf}}%
    \put(0.30942802,0.43308706){\makebox(0,0)[lt]{\lineheight{1.25}\smash{\begin{tabular}[t]{l}hybrid enhanced Heun\end{tabular}}}}%
    \put(0,0){\includegraphics[width=\unitlength,page=17]{comparison_results_median_Error_vs_time.pdf}}%
    \put(0.6533829,0.46084448){\makebox(0,0)[lt]{\lineheight{1.25}\smash{\begin{tabular}[t]{l}RK3\end{tabular}}}}%
    \put(0,0){\includegraphics[width=\unitlength,page=18]{comparison_results_median_Error_vs_time.pdf}}%
    \put(0.6533829,0.43308706){\makebox(0,0)[lt]{\lineheight{1.25}\smash{\begin{tabular}[t]{l}Heun + Richardson\end{tabular}}}}%
  \end{picture}%
\endgroup%

%% file: figures/results_in_training_area/comparison_results_box_plot1.pdf_tex
\begingroup%
  \makeatletter%
  \providecommand\color[2][]{%
    \errmessage{(Inkscape) Color is used for the text in Inkscape, but the package 'color.sty' is not loaded}%
    \renewcommand\color[2][]{}%
  }%
  \providecommand\transparent[1]{%
    \errmessage{(Inkscape) Transparency is used (non-zero) for the text in Inkscape, but the package 'transparent.sty' is not loaded}%
    \renewcommand\transparent[1]{}%
  }%
  \providecommand\rotatebox[2]{#2}%
  \newcommand*\fsize{\dimexpr\f@size pt\relax}%
  \newcommand*\lineheight[1]{\fontsize{\fsize}{#1\fsize}\selectfont}%
  \ifx\svgwidth\undefined%
    \setlength{\unitlength}{527.91497803bp}%
    \ifx\svgscale\undefined%
      \relax%
    \else%
      \setlength{\unitlength}{\unitlength * \real{\svgscale}}%
    \fi%
  \else%
    \setlength{\unitlength}{\svgwidth}%
  \fi%
  \global\let\svgwidth\undefined%
  \global\let\svgscale\undefined%
  \makeatother%
  \begin{picture}(1,0.84152756)%
    \lineheight{1}%
    \setlength\tabcolsep{0pt}%
    \put(0,0){\includegraphics[width=\unitlength,page=1]{comparison_results_box_plot1.pdf}}%
    \put(0.16179836,0.0539055){\makebox(0,0)[t]{\lineheight{1.25}\smash{\begin{tabular}[t]{c}0.30\end{tabular}}}}%
    \put(0,0){\includegraphics[width=\unitlength,page=2]{comparison_results_box_plot1.pdf}}%
    \put(0.26652417,0.0539055){\makebox(0,0)[t]{\lineheight{1.25}\smash{\begin{tabular}[t]{c}0.35\end{tabular}}}}%
    \put(0,0){\includegraphics[width=\unitlength,page=3]{comparison_results_box_plot1.pdf}}%
    \put(0.37125,0.0539055){\makebox(0,0)[t]{\lineheight{1.25}\smash{\begin{tabular}[t]{c}0.40\end{tabular}}}}%
    \put(0,0){\includegraphics[width=\unitlength,page=4]{comparison_results_box_plot1.pdf}}%
    \put(0.47597579,0.0539055){\makebox(0,0)[t]{\lineheight{1.25}\smash{\begin{tabular}[t]{c}0.45\end{tabular}}}}%
    \put(0,0){\includegraphics[width=\unitlength,page=5]{comparison_results_box_plot1.pdf}}%
    \put(0.58070165,0.0539055){\makebox(0,0)[t]{\lineheight{1.25}\smash{\begin{tabular}[t]{c}0.50\end{tabular}}}}%
    \put(0,0){\includegraphics[width=\unitlength,page=6]{comparison_results_box_plot1.pdf}}%
    \put(0.68542744,0.0539055){\makebox(0,0)[t]{\lineheight{1.25}\smash{\begin{tabular}[t]{c}0.55\end{tabular}}}}%
    \put(0,0){\includegraphics[width=\unitlength,page=7]{comparison_results_box_plot1.pdf}}%
    \put(0.79015324,0.0539055){\makebox(0,0)[t]{\lineheight{1.25}\smash{\begin{tabular}[t]{c}0.60\end{tabular}}}}%
    \put(0,0){\includegraphics[width=\unitlength,page=8]{comparison_results_box_plot1.pdf}}%
    \put(0.89487909,0.0539055){\makebox(0,0)[t]{\lineheight{1.25}\smash{\begin{tabular}[t]{c}0.65\end{tabular}}}}%
    \put(0.54279097,0.0207846){\makebox(0,0)[t]{\lineheight{1.25}\smash{\begin{tabular}[t]{c}$\delta_t$\end{tabular}}}}%
    \put(0,0){\includegraphics[width=\unitlength,page=9]{comparison_results_box_plot1.pdf}}%
    \put(0.10673594,0.13064656){\makebox(0,0)[rt]{\lineheight{1.25}\smash{\begin{tabular}[t]{r}-8.6\end{tabular}}}}%
    \put(0,0){\includegraphics[width=\unitlength,page=10]{comparison_results_box_plot1.pdf}}%
    \put(0.10673594,0.21837052){\makebox(0,0)[rt]{\lineheight{1.25}\smash{\begin{tabular}[t]{r}-8.4\end{tabular}}}}%
    \put(0,0){\includegraphics[width=\unitlength,page=11]{comparison_results_box_plot1.pdf}}%
    \put(0.10673594,0.30609447){\makebox(0,0)[rt]{\lineheight{1.25}\smash{\begin{tabular}[t]{r}-8.2\end{tabular}}}}%
    \put(0,0){\includegraphics[width=\unitlength,page=12]{comparison_results_box_plot1.pdf}}%
    \put(0.10673594,0.39381843){\makebox(0,0)[rt]{\lineheight{1.25}\smash{\begin{tabular}[t]{r}-8.0\end{tabular}}}}%
    \put(0,0){\includegraphics[width=\unitlength,page=13]{comparison_results_box_plot1.pdf}}%
    \put(0.10673594,0.48154242){\makebox(0,0)[rt]{\lineheight{1.25}\smash{\begin{tabular}[t]{r}-7.8\end{tabular}}}}%
    \put(0,0){\includegraphics[width=\unitlength,page=14]{comparison_results_box_plot1.pdf}}%
    \put(0.10673594,0.56926638){\makebox(0,0)[rt]{\lineheight{1.25}\smash{\begin{tabular}[t]{r}-7.6\end{tabular}}}}%
    \put(0,0){\includegraphics[width=\unitlength,page=15]{comparison_results_box_plot1.pdf}}%
    \put(0.10673594,0.65699032){\makebox(0,0)[rt]{\lineheight{1.25}\smash{\begin{tabular}[t]{r}-7.4\end{tabular}}}}%
    \put(0.03666785,0.40524519){\rotatebox{90}{\makebox(0,0)[t]{\lineheight{1.25}\smash{\begin{tabular}[t]{c}$\log_{10}(\delta_e)$\end{tabular}}}}}%
    \put(0,0){\includegraphics[width=\unitlength,page=16]{comparison_results_box_plot1.pdf}}%
    \put(0.17499029,0.79273656){\makebox(0,0)[lt]{\lineheight{1.25}\smash{\begin{tabular}[t]{l}enhanced Heun\end{tabular}}}}%
    \put(0,0){\includegraphics[width=\unitlength,page=17]{comparison_results_box_plot1.pdf}}%
    \put(0.17499029,0.73309628){\makebox(0,0)[lt]{\lineheight{1.25}\smash{\begin{tabular}[t]{l}hybrid enhanced Heun\end{tabular}}}}%
    \put(0,0){\includegraphics[width=\unitlength,page=18]{comparison_results_box_plot1.pdf}}%
    \put(0.91402028,0.79273656){\makebox(0,0)[lt]{\lineheight{1.25}\smash{\begin{tabular}[t]{l}RK3\end{tabular}}}}%
  \end{picture}%
\endgroup%

%% file: figures/results_in_training_area/comparison_results_box_plot2.pdf_tex
\begingroup%
  \makeatletter%
  \providecommand\color[2][]{%
    \errmessage{(Inkscape) Color is used for the text in Inkscape, but the package 'color.sty' is not loaded}%
    \renewcommand\color[2][]{}%
  }%
  \providecommand\transparent[1]{%
    \errmessage{(Inkscape) Transparency is used (non-zero) for the text in Inkscape, but the package 'transparent.sty' is not loaded}%
    \renewcommand\transparent[1]{}%
  }%
  \providecommand\rotatebox[2]{#2}%
  \newcommand*\fsize{\dimexpr\f@size pt\relax}%
  \newcommand*\lineheight[1]{\fontsize{\fsize}{#1\fsize}\selectfont}%
  \ifx\svgwidth\undefined%
    \setlength{\unitlength}{548.27502441bp}%
    \ifx\svgscale\undefined%
      \relax%
    \else%
      \setlength{\unitlength}{\unitlength * \real{\svgscale}}%
    \fi%
  \else%
    \setlength{\unitlength}{\svgwidth}%
  \fi%
  \global\let\svgwidth\undefined%
  \global\let\svgscale\undefined%
  \makeatother%
  \begin{picture}(1,0.81027766)%
    \lineheight{1}%
    \setlength\tabcolsep{0pt}%
    \put(0,0){\includegraphics[width=\unitlength,page=1]{comparison_results_box_plot2.pdf}}%
    \put(0.15725092,0.05190373){\makebox(0,0)[t]{\lineheight{1.25}\smash{\begin{tabular}[t]{c}0.6\end{tabular}}}}%
    \put(0,0){\includegraphics[width=\unitlength,page=2]{comparison_results_box_plot2.pdf}}%
    \put(0.27079637,0.05190373){\makebox(0,0)[t]{\lineheight{1.25}\smash{\begin{tabular}[t]{c}0.7\end{tabular}}}}%
    \put(0,0){\includegraphics[width=\unitlength,page=3]{comparison_results_box_plot2.pdf}}%
    \put(0.38434182,0.05190373){\makebox(0,0)[t]{\lineheight{1.25}\smash{\begin{tabular}[t]{c}0.8\end{tabular}}}}%
    \put(0,0){\includegraphics[width=\unitlength,page=4]{comparison_results_box_plot2.pdf}}%
    \put(0.49788727,0.05190373){\makebox(0,0)[t]{\lineheight{1.25}\smash{\begin{tabular}[t]{c}0.9\end{tabular}}}}%
    \put(0,0){\includegraphics[width=\unitlength,page=5]{comparison_results_box_plot2.pdf}}%
    \put(0.6114327,0.05190373){\makebox(0,0)[t]{\lineheight{1.25}\smash{\begin{tabular}[t]{c}1.0\end{tabular}}}}%
    \put(0,0){\includegraphics[width=\unitlength,page=6]{comparison_results_box_plot2.pdf}}%
    \put(0.72497818,0.05190373){\makebox(0,0)[t]{\lineheight{1.25}\smash{\begin{tabular}[t]{c}1.1\end{tabular}}}}%
    \put(0,0){\includegraphics[width=\unitlength,page=7]{comparison_results_box_plot2.pdf}}%
    \put(0.8385236,0.05190373){\makebox(0,0)[t]{\lineheight{1.25}\smash{\begin{tabular}[t]{c}1.2\end{tabular}}}}%
    \put(0,0){\includegraphics[width=\unitlength,page=8]{comparison_results_box_plot2.pdf}}%
    \put(0.95206902,0.05190373){\makebox(0,0)[t]{\lineheight{1.25}\smash{\begin{tabular}[t]{c}1.3\end{tabular}}}}%
    \put(0.55976925,0.02001277){\makebox(0,0)[t]{\lineheight{1.25}\smash{\begin{tabular}[t]{c}$\delta_t$\end{tabular}}}}%
    \put(0,0){\includegraphics[width=\unitlength,page=9]{comparison_results_box_plot2.pdf}}%
    \put(0.13990697,0.10531437){\makebox(0,0)[rt]{\lineheight{1.25}\smash{\begin{tabular}[t]{r}-10.25\end{tabular}}}}%
    \put(0,0){\includegraphics[width=\unitlength,page=10]{comparison_results_box_plot2.pdf}}%
    \put(0.13990697,0.18737183){\makebox(0,0)[rt]{\lineheight{1.25}\smash{\begin{tabular}[t]{r}-10.00\end{tabular}}}}%
    \put(0,0){\includegraphics[width=\unitlength,page=11]{comparison_results_box_plot2.pdf}}%
    \put(0.13990697,0.26942933){\makebox(0,0)[rt]{\lineheight{1.25}\smash{\begin{tabular}[t]{r}-9.75\end{tabular}}}}%
    \put(0,0){\includegraphics[width=\unitlength,page=12]{comparison_results_box_plot2.pdf}}%
    \put(0.13990697,0.35148679){\makebox(0,0)[rt]{\lineheight{1.25}\smash{\begin{tabular}[t]{r}-9.50\end{tabular}}}}%
    \put(0,0){\includegraphics[width=\unitlength,page=13]{comparison_results_box_plot2.pdf}}%
    \put(0.13990697,0.43354426){\makebox(0,0)[rt]{\lineheight{1.25}\smash{\begin{tabular}[t]{r}-9.25\end{tabular}}}}%
    \put(0,0){\includegraphics[width=\unitlength,page=14]{comparison_results_box_plot2.pdf}}%
    \put(0.13990697,0.51560174){\makebox(0,0)[rt]{\lineheight{1.25}\smash{\begin{tabular}[t]{r}-9.00\end{tabular}}}}%
    \put(0,0){\includegraphics[width=\unitlength,page=15]{comparison_results_box_plot2.pdf}}%
    \put(0.13990697,0.59765924){\makebox(0,0)[rt]{\lineheight{1.25}\smash{\begin{tabular}[t]{r}-8.75\end{tabular}}}}%
    \put(0,0){\includegraphics[width=\unitlength,page=16]{comparison_results_box_plot2.pdf}}%
    \put(0.13990697,0.67971672){\makebox(0,0)[rt]{\lineheight{1.25}\smash{\begin{tabular}[t]{r}-8.50\end{tabular}}}}%
    \put(0.0353062,0.39019652){\rotatebox{90}{\makebox(0,0)[t]{\lineheight{1.25}\smash{\begin{tabular}[t]{c}$\log_{10}(\delta_e)$\end{tabular}}}}}%
    \put(0,0){\includegraphics[width=\unitlength,page=17]{comparison_results_box_plot2.pdf}}%
    \put(0.20562673,0.7632985){\makebox(0,0)[lt]{\lineheight{1.25}\smash{\begin{tabular}[t]{l}enhanced Heun\end{tabular}}}}%
    \put(0,0){\includegraphics[width=\unitlength,page=18]{comparison_results_box_plot2.pdf}}%
    \put(0.20562673,0.70587294){\makebox(0,0)[lt]{\lineheight{1.25}\smash{\begin{tabular}[t]{l}hybrid enhanced Heun\end{tabular}}}}%
    \put(0,0){\includegraphics[width=\unitlength,page=19]{comparison_results_box_plot2.pdf}}%
    \put(0.91721305,0.7632985){\makebox(0,0)[lt]{\lineheight{1.25}\smash{\begin{tabular}[t]{l}RK3\end{tabular}}}}%
  \end{picture}%
\endgroup%

%% file: figures/results_far_from_training_data/comparison_results_box_plot1_far.pdf_tex
\begingroup%
  \makeatletter%
  \providecommand\color[2][]{%
    \errmessage{(Inkscape) Color is used for the text in Inkscape, but the package 'color.sty' is not loaded}%
    \renewcommand\color[2][]{}%
  }%
  \providecommand\transparent[1]{%
    \errmessage{(Inkscape) Transparency is used (non-zero) for the text in Inkscape, but the package 'transparent.sty' is not loaded}%
    \renewcommand\transparent[1]{}%
  }%
  \providecommand\rotatebox[2]{#2}%
  \newcommand*\fsize{\dimexpr\f@size pt\relax}%
  \newcommand*\lineheight[1]{\fontsize{\fsize}{#1\fsize}\selectfont}%
  \ifx\svgwidth\undefined%
    \setlength{\unitlength}{538.0949707bp}%
    \ifx\svgscale\undefined%
      \relax%
    \else%
      \setlength{\unitlength}{\unitlength * \real{\svgscale}}%
    \fi%
  \else%
    \setlength{\unitlength}{\svgwidth}%
  \fi%
  \global\let\svgwidth\undefined%
  \global\let\svgscale\undefined%
  \makeatother%
  \begin{picture}(1,0.82560706)%
    \lineheight{1}%
    \setlength\tabcolsep{0pt}%
    \put(0,0){\includegraphics[width=\unitlength,page=1]{comparison_results_box_plot1_far.pdf}}%
    \put(0.18001028,0.05288568){\makebox(0,0)[t]{\lineheight{1.25}\smash{\begin{tabular}[t]{c}0.30\end{tabular}}}}%
    \put(0,0){\includegraphics[width=\unitlength,page=2]{comparison_results_box_plot1_far.pdf}}%
    \put(0.2848749,0.05288568){\makebox(0,0)[t]{\lineheight{1.25}\smash{\begin{tabular}[t]{c}0.35\end{tabular}}}}%
    \put(0,0){\includegraphics[width=\unitlength,page=3]{comparison_results_box_plot1_far.pdf}}%
    \put(0.38973949,0.05288568){\makebox(0,0)[t]{\lineheight{1.25}\smash{\begin{tabular}[t]{c}0.40\end{tabular}}}}%
    \put(0,0){\includegraphics[width=\unitlength,page=4]{comparison_results_box_plot1_far.pdf}}%
    \put(0.49460411,0.05288568){\makebox(0,0)[t]{\lineheight{1.25}\smash{\begin{tabular}[t]{c}0.45\end{tabular}}}}%
    \put(0,0){\includegraphics[width=\unitlength,page=5]{comparison_results_box_plot1_far.pdf}}%
    \put(0.59946872,0.05288568){\makebox(0,0)[t]{\lineheight{1.25}\smash{\begin{tabular}[t]{c}0.50\end{tabular}}}}%
    \put(0,0){\includegraphics[width=\unitlength,page=6]{comparison_results_box_plot1_far.pdf}}%
    \put(0.70433329,0.05288568){\makebox(0,0)[t]{\lineheight{1.25}\smash{\begin{tabular}[t]{c}0.55\end{tabular}}}}%
    \put(0,0){\includegraphics[width=\unitlength,page=7]{comparison_results_box_plot1_far.pdf}}%
    \put(0.80919791,0.05288568){\makebox(0,0)[t]{\lineheight{1.25}\smash{\begin{tabular}[t]{c}0.60\end{tabular}}}}%
    \put(0,0){\includegraphics[width=\unitlength,page=8]{comparison_results_box_plot1_far.pdf}}%
    \put(0.91406252,0.05288568){\makebox(0,0)[t]{\lineheight{1.25}\smash{\begin{tabular}[t]{c}0.65\end{tabular}}}}%
    \put(0.55144078,0.02039139){\makebox(0,0)[t]{\lineheight{1.25}\smash{\begin{tabular}[t]{c}$\delta_t$\end{tabular}}}}%
    \put(0,0){\includegraphics[width=\unitlength,page=9]{comparison_results_box_plot1_far.pdf}}%
    \put(0.12363523,0.14034474){\makebox(0,0)[rt]{\lineheight{1.25}\smash{\begin{tabular}[t]{r}-8.50\end{tabular}}}}%
    \put(0,0){\includegraphics[width=\unitlength,page=10]{comparison_results_box_plot1_far.pdf}}%
    \put(0.12363523,0.21017654){\makebox(0,0)[rt]{\lineheight{1.25}\smash{\begin{tabular}[t]{r}-8.25\end{tabular}}}}%
    \put(0,0){\includegraphics[width=\unitlength,page=11]{comparison_results_box_plot1_far.pdf}}%
    \put(0.12363523,0.28000834){\makebox(0,0)[rt]{\lineheight{1.25}\smash{\begin{tabular}[t]{r}-8.00\end{tabular}}}}%
    \put(0,0){\includegraphics[width=\unitlength,page=12]{comparison_results_box_plot1_far.pdf}}%
    \put(0.12363523,0.34984015){\makebox(0,0)[rt]{\lineheight{1.25}\smash{\begin{tabular}[t]{r}-7.75\end{tabular}}}}%
    \put(0,0){\includegraphics[width=\unitlength,page=13]{comparison_results_box_plot1_far.pdf}}%
    \put(0.12363523,0.41967198){\makebox(0,0)[rt]{\lineheight{1.25}\smash{\begin{tabular}[t]{r}-7.50\end{tabular}}}}%
    \put(0,0){\includegraphics[width=\unitlength,page=14]{comparison_results_box_plot1_far.pdf}}%
    \put(0.12363523,0.48950378){\makebox(0,0)[rt]{\lineheight{1.25}\smash{\begin{tabular}[t]{r}-7.25\end{tabular}}}}%
    \put(0,0){\includegraphics[width=\unitlength,page=15]{comparison_results_box_plot1_far.pdf}}%
    \put(0.12363523,0.55933561){\makebox(0,0)[rt]{\lineheight{1.25}\smash{\begin{tabular}[t]{r}-7.00\end{tabular}}}}%
    \put(0,0){\includegraphics[width=\unitlength,page=16]{comparison_results_box_plot1_far.pdf}}%
    \put(0.12363523,0.62916741){\makebox(0,0)[rt]{\lineheight{1.25}\smash{\begin{tabular}[t]{r}-6.75\end{tabular}}}}%
    \put(0.03597415,0.39757853){\rotatebox{90}{\makebox(0,0)[t]{\lineheight{1.25}\smash{\begin{tabular}[t]{c}$\log_{10}(\delta_e)$\end{tabular}}}}}%
    \put(0,0){\includegraphics[width=\unitlength,page=17]{comparison_results_box_plot1_far.pdf}}%
    \put(0.19059832,0.77773911){\makebox(0,0)[lt]{\lineheight{1.25}\smash{\begin{tabular}[t]{l}enhanced Heun\end{tabular}}}}%
    \put(0,0){\includegraphics[width=\unitlength,page=18]{comparison_results_box_plot1_far.pdf}}%
    \put(0.19059832,0.71922714){\makebox(0,0)[lt]{\lineheight{1.25}\smash{\begin{tabular}[t]{l}hybrid enhanced Heun\end{tabular}}}}%
    \put(0,0){\includegraphics[width=\unitlength,page=19]{comparison_results_box_plot1_far.pdf}}%
    \put(0.91564689,0.77773911){\makebox(0,0)[lt]{\lineheight{1.25}\smash{\begin{tabular}[t]{l}RK3\end{tabular}}}}%
  \end{picture}%
\endgroup%

%% file: figures/results_far_from_training_data/comparison_results_box_plot2_far.pdf_tex
\begingroup%
  \makeatletter%
  \providecommand\color[2][]{%
    \errmessage{(Inkscape) Color is used for the text in Inkscape, but the package 'color.sty' is not loaded}%
    \renewcommand\color[2][]{}%
  }%
  \providecommand\transparent[1]{%
    \errmessage{(Inkscape) Transparency is used (non-zero) for the text in Inkscape, but the package 'transparent.sty' is not loaded}%
    \renewcommand\transparent[1]{}%
  }%
  \providecommand\rotatebox[2]{#2}%
  \newcommand*\fsize{\dimexpr\f@size pt\relax}%
  \newcommand*\lineheight[1]{\fontsize{\fsize}{#1\fsize}\selectfont}%
  \ifx\svgwidth\undefined%
    \setlength{\unitlength}{548.27502441bp}%
    \ifx\svgscale\undefined%
      \relax%
    \else%
      \setlength{\unitlength}{\unitlength * \real{\svgscale}}%
    \fi%
  \else%
    \setlength{\unitlength}{\svgwidth}%
  \fi%
  \global\let\svgwidth\undefined%
  \global\let\svgscale\undefined%
  \makeatother%
  \begin{picture}(1,0.81027766)%
    \lineheight{1}%
    \setlength\tabcolsep{0pt}%
    \put(0,0){\includegraphics[width=\unitlength,page=1]{comparison_results_box_plot2_far.pdf}}%
    \put(0.256329,0.05190373){\makebox(0,0)[t]{\lineheight{1.25}\smash{\begin{tabular}[t]{c}0.7\end{tabular}}}}%
    \put(0,0){\includegraphics[width=\unitlength,page=2]{comparison_results_box_plot2_far.pdf}}%
    \put(0.36693229,0.05190373){\makebox(0,0)[t]{\lineheight{1.25}\smash{\begin{tabular}[t]{c}0.8\end{tabular}}}}%
    \put(0,0){\includegraphics[width=\unitlength,page=3]{comparison_results_box_plot2_far.pdf}}%
    \put(0.47753558,0.05190373){\makebox(0,0)[t]{\lineheight{1.25}\smash{\begin{tabular}[t]{c}0.9\end{tabular}}}}%
    \put(0,0){\includegraphics[width=\unitlength,page=4]{comparison_results_box_plot2_far.pdf}}%
    \put(0.58813881,0.05190373){\makebox(0,0)[t]{\lineheight{1.25}\smash{\begin{tabular}[t]{c}1.0\end{tabular}}}}%
    \put(0,0){\includegraphics[width=\unitlength,page=5]{comparison_results_box_plot2_far.pdf}}%
    \put(0.6987421,0.05190373){\makebox(0,0)[t]{\lineheight{1.25}\smash{\begin{tabular}[t]{c}1.1\end{tabular}}}}%
    \put(0,0){\includegraphics[width=\unitlength,page=6]{comparison_results_box_plot2_far.pdf}}%
    \put(0.80934539,0.05190373){\makebox(0,0)[t]{\lineheight{1.25}\smash{\begin{tabular}[t]{c}1.2\end{tabular}}}}%
    \put(0,0){\includegraphics[width=\unitlength,page=7]{comparison_results_box_plot2_far.pdf}}%
    \put(0.91994868,0.05190373){\makebox(0,0)[t]{\lineheight{1.25}\smash{\begin{tabular}[t]{c}1.3\end{tabular}}}}%
    \put(0.55976925,0.02001277){\makebox(0,0)[t]{\lineheight{1.25}\smash{\begin{tabular}[t]{c}$\delta_t$\end{tabular}}}}%
    \put(0,0){\includegraphics[width=\unitlength,page=8]{comparison_results_box_plot2_far.pdf}}%
    \put(0.13990697,0.10134145){\makebox(0,0)[rt]{\lineheight{1.25}\smash{\begin{tabular}[t]{r}-10.25\end{tabular}}}}%
    \put(0,0){\includegraphics[width=\unitlength,page=9]{comparison_results_box_plot2_far.pdf}}%
    \put(0.13990697,0.16958995){\makebox(0,0)[rt]{\lineheight{1.25}\smash{\begin{tabular}[t]{r}-10.00\end{tabular}}}}%
    \put(0,0){\includegraphics[width=\unitlength,page=10]{comparison_results_box_plot2_far.pdf}}%
    \put(0.13990697,0.23783844){\makebox(0,0)[rt]{\lineheight{1.25}\smash{\begin{tabular}[t]{r}-9.75\end{tabular}}}}%
    \put(0,0){\includegraphics[width=\unitlength,page=11]{comparison_results_box_plot2_far.pdf}}%
    \put(0.13990697,0.30608694){\makebox(0,0)[rt]{\lineheight{1.25}\smash{\begin{tabular}[t]{r}-9.50\end{tabular}}}}%
    \put(0,0){\includegraphics[width=\unitlength,page=12]{comparison_results_box_plot2_far.pdf}}%
    \put(0.13990697,0.37433541){\makebox(0,0)[rt]{\lineheight{1.25}\smash{\begin{tabular}[t]{r}-9.25\end{tabular}}}}%
    \put(0,0){\includegraphics[width=\unitlength,page=13]{comparison_results_box_plot2_far.pdf}}%
    \put(0.13990697,0.44258387){\makebox(0,0)[rt]{\lineheight{1.25}\smash{\begin{tabular}[t]{r}-9.00\end{tabular}}}}%
    \put(0,0){\includegraphics[width=\unitlength,page=14]{comparison_results_box_plot2_far.pdf}}%
    \put(0.13990697,0.51083237){\makebox(0,0)[rt]{\lineheight{1.25}\smash{\begin{tabular}[t]{r}-8.75\end{tabular}}}}%
    \put(0,0){\includegraphics[width=\unitlength,page=15]{comparison_results_box_plot2_far.pdf}}%
    \put(0.13990697,0.57908086){\makebox(0,0)[rt]{\lineheight{1.25}\smash{\begin{tabular}[t]{r}-8.50\end{tabular}}}}%
    \put(0,0){\includegraphics[width=\unitlength,page=16]{comparison_results_box_plot2_far.pdf}}%
    \put(0.13990697,0.64732935){\makebox(0,0)[rt]{\lineheight{1.25}\smash{\begin{tabular}[t]{r}-8.25\end{tabular}}}}%
    \put(0.0353062,0.39019652){\rotatebox{90}{\makebox(0,0)[t]{\lineheight{1.25}\smash{\begin{tabular}[t]{c}$\log_{10}(\delta_e)$\end{tabular}}}}}%
    \put(0,0){\includegraphics[width=\unitlength,page=17]{comparison_results_box_plot2_far.pdf}}%
    \put(0.20562673,0.7632985){\makebox(0,0)[lt]{\lineheight{1.25}\smash{\begin{tabular}[t]{l}enhanced Heun\end{tabular}}}}%
    \put(0,0){\includegraphics[width=\unitlength,page=18]{comparison_results_box_plot2_far.pdf}}%
    \put(0.20562673,0.70587294){\makebox(0,0)[lt]{\lineheight{1.25}\smash{\begin{tabular}[t]{l}hybrid enhanced Heun\end{tabular}}}}%
    \put(0,0){\includegraphics[width=\unitlength,page=19]{comparison_results_box_plot2_far.pdf}}%
    \put(0.91721305,0.7632985){\makebox(0,0)[lt]{\lineheight{1.25}\smash{\begin{tabular}[t]{l}RK3\end{tabular}}}}%
  \end{picture}%
\endgroup%